\documentclass[a4paper,11pt,titlepage]{article}

\usepackage{amssymb,amsmath, amsfonts, latexsym,paralist}
\usepackage{graphics, epsfig,times}
\newtheorem{lemma}{Lemma}
\newtheorem{result}{Result}

\newtheorem{theorem}{Theorem}
\newtheorem{proposition}{Proposition}
\newtheorem{assumption}{Assumption}
\newcommand{\eq}[1]{(\ref{#1})}

\newcommand{\fig}[1]{Figure~\ref{fig:#1}}
\newcommand{\tab}[1]{Table~\ref{tab:#1}}
\newcommand{\thm}[1]{Theorem~\ref{thm:#1}}

\newenvironment{remark}
   {\vskip 0.2ex minus 0.3ex\trivlist\item[\hskip
     \labelsep{\sc Remark:}]}
   {\endtrivlist\vskip 0.2ex minus 0.3ex}
\newcommand{\qed}{ \hfill {$\Box$}\\%
           }
\newenvironment{proof}
   {\vskip 0.5ex minus 0.3ex {\bfseries Proof:} }
{\qed}
\font\cmss=cmss10 scaled 1100
\font\goth=eufm10 at 11pt
\newcommand{\bsl}{\bfseries\itshape}  
\def\ds{\displaystyle}
\def\Field{{\hbox{\goth F}}}
\def\Real{{\mathbb R}}   % Real numbers
\def\Integer{\mathbb N}  % Integer numbers
\def\var{\hbox{\cmss Var}}

\def\given{{\: | \: }}
\def\part#1{{\partial\over\partial{#1}}}
\def\ind#1{{\mathbf 1}_{\left\{{#1}\right\}} }  

\def\silent#1{}
\def\param{{\alpha}}
\def \a {{\alpha}}

\def\th{{\theta}}
\def \p {{\prime}}

\def\prob{{\mathbb P}}
\def\esp{{\mathbb E}}

\def \Epi {{\esp}_{\pi(\param)}}
\def\covpi{{\hbox{\rm Cov}}_{\pi(\param)}}
\def\PP{\hbox{\cmss P}}
\def\E{\hbox{\cmss E}_i}       
\def \Ep {{\esp}_{\param}}
\def\nablap{{\nabla_\param}}
\def \nablal{{\nabla_\a}}
 \def \G{{B}}

\def\ep{{\epsilon}}

\def\actionset{{\cal U}}
\def \admissible{{\cal D}}

\def\la{{\lambda}}
\newcommand{\Al}{{\large\text{\boldmath$\a$}}}
\def\Param{{\Al}}
\def \con {{\beta}}
\def \rcon {{\gamma}}
\def \defn {\stackrel{\triangle}{=}}
\def \policy {\mathbf u}
\def \La {{\cal L}}

\def \GF{{\widehat{\nabla C}^{\text{Frozen}}}}
\def \GS{{\widehat{\nabla C}^{\text{Score}}}}
\def \tp {{P}}
\def \Z {{\cal Z}}

\def \J {{\mathcal I}}
\def \q {\mathcal{C}}
\newcommand{\belief}{\pi}
\newcommand{\tpm}{P_\theta}
\newcommand{\dob}{\rho}
\newcommand{\bsize}{N}
\newcommand{\dvar}[2]{\|{#1}-{#2}\|_{\text{\tiny{TV}}}}
\author{Vikram Krishnamurthy, \\ 
Cornell Tech \& School of
Electrical and Computer Engineering\\ 
Cornell University
\\ email: {\tt vikramk@cornell.edu}
\and
Felisa J.~V\'azquez Abad,\\ 
Hunter College, City University of New York\\
email: {\tt felisav@hunter.cuny.edu}}

\title{Real-Time Reinforcement Learning of Constrained Markov Decision  Processes
  with Weak Derivatives}

\begin{document}
\maketitle

\begin{abstract}
  We present  on-line policy gradient algorithms for computing the
  locally optimal policy of a constrained, average cost, finite state
  Markov Decision Process.
% Because the optimal control strategy is known to be a
%randomized policy, we consider here a parameterization of the action probabilities to
%establish the optimization problem. 
The stochastic approximation
algorithms
  require estimation of the gradient of the cost function with
  respect to the parameter that characterizes the randomized policy.
  We 
propose a spherical coordinate  parameterization
and present a novel simulation based gradient estimation
scheme involving weak derivatives (measure-valued differentiation).
Such methods have substantially reduced variance compared to the widely used score function method. Similar to neuro-dynamic
  programming algorithms (e.g. Q-learning or Temporal Difference
  methods), the algorithms proposed in this paper are simulation based
  and do not require explicit knowledge of the underlying parameters
  such as transition probabilities.  However, unlike neuro-dynamic
  programming methods, the algorithms proposed here can handle
  constraints and time varying parameters. Numerical examples are
  given to illustrate the performance of the algorithms.

  This paper was originally written in 2004. One reason we are putting this on arxiv now is that  the score function gradient estimator  continues to be used in the online  reinforcement learning literature  even though its variance grows as $O(n)$ given $n$ data points (for a Markov process). In comparison the weak derivative estimator has significantly smaller variance of $O(1)$ as reported in this paper (and elsewhere).
  
\end{abstract}

\section{Introduction} \label{sec:intro}

This paper deals with the
adaptive control of  a finite-state finite-action average
cost constrained Markov
Decision process (MDP). Such a constrained  MDP is constructed as follows:

\paragraph{Constrained MDP.}
Let $S$ denote an arbitrary finite set called the {\em state space}.
Let $\actionset_i$, $i\in S$ denote an arbitrary  collection of finite sets
called  {\em action sets}.
A unichain Markov Decision Process \cite{Put94}
(MDP) $\{X_n\}$ with finite state space $S$ evolves as follows:
  When the system is in state $i\in S$, a
finite number of possible actions, which are elements
of the  finite set $\actionset_i$
 can be taken. Let $u_n$ denote
the action taken by the decision maker at time $n$ and let $d(i)+1$
 denote
the cardinality of the action set $\actionset_i$. 
 The evolution of the system is Markovian with a 
transition probability matrix $A(u)$ that
depends on the action $u\in \actionset_i$, that is for $i,j \in S$,
\begin{equation}
\label{Aij}
A_{ij}(u) \defn \prob[X_{n+1} = j | X_n = i, u_n = u],  \quad u \in \actionset_i,
\quad n=0,1,\ldots
\end{equation}
Denote by  $\Field_n, n\ge 1$  the 
$\sigma$-algebra generated by the observed {\em system trajectory}
$(X_{0}, \ldots,X_n, u_0,\ldots,u_{n-1})$ and set
$\Field_{0}$ as the $\sigma$-algebra generated by $X_{0}$.  The
filtration of the process is the increasing sequence of $\sigma$-algebras
$\{\Field_n, n\ge0\}$.  
Define the set of {\bf admissible} policies $
\admissible = \{ \policy = \{u_n\} : u_n \mbox { is measurable w.r.t. } \Field_n, \
\forall n\in \Integer\}
$. This 
 means that $u_n$ is a (possibly random) function of 
$(X_0,\ldots,X_n,
u_0,\ldots,u_{n-1})$.  By {\em unichain} \cite[pp.348]{Put94}
we mean that every policy where $u_n$  is a deterministic
function of $X_n$ consists of a 
single recurrent class  plus possibly an empty set of transient
states.  

 The cost incurred at stage
$n$ is a known bounded function
$c(X_n,u_n)\geq 0$ where $c: S \times \actionset \rightarrow \Real$. For any
admissible policy $\policy \in \admissible$,
let $ \esp_{\policy}$ denote the corresponding expectation and 
 define the infinite horizon average cost
\begin{equation}
   \label{J}
J_{x_0}(\policy) = \lim_{N\to\infty} \sup {1\over N} \esp_{\policy} \left[\sum_{n=1}^N
c(X_n,u_n) \mid X_0 = x_0\right].
\end{equation}
Motivated by several problems in telecommunication network optimization such
as admission control in wireless networks \cite{SKP02},
 we consider the cost  (\ref{J}), subject to  $L$ 
 sample path constraints  
$$ \prob_{\policy}\left[\lim \sup_{N \rightarrow \infty} \frac{1}{N} \sum_{n=1}^N
\con_l(X_n,u_n)
\leq \rcon_l \right] = 1, \quad l=1,2,\ldots,L,$$
where $\con_l: S \times \actionset \rightarrow \Real$ are known bounded functions and  $\gamma_l$ are known constants.
These are used, for example, in admission control of telecommunication
networks to depict quality of service (QoS) constraints,
see \cite{SKP02}.
For unichain MDPs, these sample path constraints are equivalent to the
average constraints  \cite{RV89}:
\begin{equation}
 \label{costconstraint}
   \lim_{N\to\infty} \sup {1\over N} \esp_{\policy}\left[ \sum_{n=1}^N\con_l(X_n,u_n)
\right] \leq
\rcon_l,
\quad l=1,\ldots,L.
\end{equation}
The aim is to compute the optimal policy $\policy^* \in \admissible$ that satisfies
\begin{equation}
J_{x_0}(\policy^*) = \inf_{\policy \in \admissible} J_{x_0}(\policy) \quad
\forall x_0 \in S, \label{eq:objective1}
\end{equation}
 i.e., $\policy^*$ has the minimum cost for all initial states
$x_0 \in S$ subject to the constraints (\ref{costconstraint}).
It is well known \cite{Alt99} that if $L>0$  then
the optimal policy $\policy^*$ 
is {\em randomized} for at most $L$ of the states.
If the transition probabilities $A_{ij}(u)$ (\ref{Aij}) are
 known, then the optimal policy $\policy^*$ for 
the constrained MDP (\ref{J}), (\ref{costconstraint}) is
straightforwardly computed as the solution of a linear programming
problem.

%% In the special case when there are no constraints
%% (i.e., $L=0$),  the
%% optimal stationary  policy  for the MDP (\ref{eq:objective1})
%% is  a {\em pure} Markovian
%% policy, i.e.,  $u^*_n$ is a deterministic
%% function of $X_n$.
%%  Thus the solution of the MDP (\ref{eq:objective1})
%% is obtained by determining
%% the function that selects the best
%% action to be taken at each state. 

\paragraph{Objectives.}

This paper presents policy gradient (stochastic approximation)  algorithms for adaptively
computing the optimal policy $\policy^*$ of the above
constrained MDP  (\ref{J}), (\ref{costconstraint})
when the transition
probabilities (\ref{Aij}) are not known -- so the problem 
(\ref{J}), (\ref{costconstraint}) is an adaptive constrained Markov
decision process problem. 
There are two methodologies that are used in the literature
for solving such stochastic adaptive control problems: {\bsl direct
  methods}, where the unknown transition probabilities $A_{ij}(u)$ 
are estimated
{\em simultaneously} while updating the control policy, and {\bsl
  implicit methods} -- such as simulation based methods, where the
transition probabilities are not directly estimated in order to
compute the control policy.

In this paper, we focus on implicit simulation-based algorithms
for solving the  MDP  (\ref{J}), (\ref{costconstraint}).
By simulation based we mean that although
the transition probabilities $A(u)$, $u \in \actionset_i$, $i\in S$
 are unknown,  the decision maker 
can observe the system trajectory under any choice of control actions $\policy = \{u_n\}$. 
In other words, the adaptive control algorithms we present
are  adapted to the filtration $\{\Field_n; n\ge 0\}$ defined above.
Moreover, these algorithms can deal with slowly time varying transition
probabilities $A(u)$.

Neurodynamic programming methods \cite{BT96}
such as Q-learning and temporal
 difference methods  are also examples of simulation based implicit methods 
that have been widely
 used to solve unconstrained MDPs where the optimal
policy $\policy^*$ is a pure policy, i.e., $u_n^* $ is a deterministic
function of $X_n$.  However, for the constrained MDP 
 (\ref{J}), (\ref{costconstraint}),
since
 the optimal policy is randomized, there seems to be no obvious way of
 modifying such neurodynamic programming
 algorithms to obtain an optimal randomized policy.

\paragraph{Summary of Main results} 
\label{sec:contribution}
As mentioned above, the optimal policy $\policy^*$  for the
 constrained MDP (\ref{J}), (\ref{costconstraint}) is 
 randomized, that is, it
corresponds to the decision maker choosing an action from the set of possible actions according to certain probabilities, 
called {\em action probabilities}. 
The constrained MDP control problem can be formulated as a stochastic optimization problem in terms of these
action probabilities. 
Our approach is to use a simulation based
 stochastic gradient  algorithm to adaptively track the optimal
action probabilities.
%solve this constrained optimization problem. 

There are two main contributions in this paper:
\medskip

\noindent{\bsl 1. Parameter Free Measure Valued (Weak Derivative) Gradient Estimation}: The most important issue in implementing the
stochastic gradient algorithm that computes the optimal policy for the
constrained MDP is to construct consistent estimates of the gradient
of the cost and constraints of the MDP with respect to the
parametrization of the action probabilities.  There are three widely
used methodologies for simulation based gradient estimation -- score
function method, sample path derivatives and weak derivatives, see
\cite{Pfl96}\footnote{Note that finite difference methods such as the Kiefer Wolfowitz
and  Simultaneous  Perturbation Stochastic Approximation (SPSA) are not considered in this
paper to be gradient
estimators}.  The Score Function method has been used in \cite{BB99,BB02}
to estimate the gradient of an unconstrained MDP with respect to the
action probabilities.  However, the score function method suffers from
large variance and can yield unbounded variance estimation for MDPs.

The main contribution in this paper 
is to present  a novel gradient estimator
involving weak derivatives computed over short batch lengths to
facilitate control of time varying MDPs.  This gradient
estimator does not require explicit knowledge of the transition probabilities
$A(u)$
 -- so for brevity  we call it a ``parameter free'' gradient
estimator. To the best of our knowledge
this gradient
estimator is new, and is in contrast to the concept of the ``realization
perturbation factors'' \cite{Cao98}, which is equivalent to computing
weak derivatives over an infinite batch length, see also \cite{Pfl96}.
The key advantage of devising  gradient
estimation algorithms over short batch lengths is
that they are amenable to adaptive control of MDPs with time varying
parameters. Moreover, these gradient estimators yield several orders
 of magnitude
reduction in variance compared to the score function method in \cite{BB99,BB02}.
 The parameter free
gradient estimation algorithms we present in Sec.\ref{sec:lcmdp}
use the concept of a
``phantom''  system \cite{Vaz99}, which is a parallel system whose evolution
is identical to the original process but with a different initial
condition.  The difference between the costs of
the nominal trajectory and the phantom trajectory 
 can then be used to determine information
about the gradient of the cost function with respect to the action
probabilities. 
 Conventionally in \cite{Dai00, Cao98}, off-line
simulations are used to generate the evolution of a phantom -- we will
call such phantoms the {\em Doeblin} phantoms.  However, we are
interested in solving an MDP with unknown and possibly time varying
parameters.  In Sec.\ref{sec:lcmdp}, we present a novel method based on
``cut and paste'' techniques which look at the past observation
history to create ``frozen phantoms''. By filtering these frozen
phantoms, we derive a parameter free consistent algorithm for
estimating gradients of the cost without explicit knowledge of the
transition probabilities
of the MDP and without requiring off-line simulations of the
system.

We also present in Sec.\ref{sec:lcmdp}  
the statistical properties 
(consistency, bias for small batch size, and efficiency) of the 
gradient estimates. For the first
time we also present estimates of the number of parallel phantoms, thus
determining the mean coupling time and the mean computational complexity.

{\bsl 2. Adaptive Control of Constrained
MDP  with time varying parameters}:
Putting
 the parameter free gradient estimation algorithm
into a  stochastic gradient
algorithm results in a new simulation based
algorithm for the adaptive control of a  constrained MDP.
There are three important results that we present in this context.

First, the algorithm we present allows simulation based adaptive control of a constrained time varying MDP.  Similar to neuro-dynamic programming
algorithms \cite{BT96} (e.g., Q-learning and Temporal Difference
methods), the algorithms proposed in this paper are simulation based
and do not require explicit knowledge of the underlying parameters of
the MDP such as transition probabilities (or equivalently invariant
distributions). However, unlike Q-learning or Temporal Difference
methods, the algorithms proposed here straightforwardly can handle
constraints.
Sec.\ref{sec:convergence} develops constant
step size  stochastic approximation
algorithms, based on the deterministic
constrained optimization  algorithms of Sec.\ref{sec:detalg} for
solving time varying  constrained MDP.

Second, because the constraints in the MDP
are in the form of long term averages that are not  
known to the decision maker, it
  is not possible to use stochastic approximation algorithms with the usual gradient projection methods in
  \cite{KC78}, for example.  We present 
 primal dual based stochastic approximation algorithms with a penalty function
and  augmented Lagrangian (multiplier)
  algorithms are presented for solving the time varying constrained MDP.
 Weak convergence of the action probability estimates to the optimal
 action probabilties is established. 
Sec.\ref{sec:tradeoff} illustrates
 the numerical performance of the algorithms on a constrained
MDP with time varying transition probabilities.

Thirdly, because the action probabilities must always add up to one
and be non negative, there is strong motivation to develop a
parametrization that ensures feasibility of the estimates generated by
the stochastic approximation algorithm at every step.  We do so by
significantly extending  an idea  in \cite{KL96b} where the square root
of the probabilities were considered and projected to the tangent
manifold. We parameterize the action probabilities of the constrained
MDP 
using {\em spherical coordinates}.
This parametrization is particularly well suited to the
MDP problem and has superior convergence rate, see discussion in
Sec.\ref{sec:approaches}. To the best of our knowledge such
a spherical coordinate parameterization has not been used in the context
of stochastic approximations. For example in \cite{BB99,BB02}, a different
parameterization is used which is similar to the generalized gradient
approach of \cite{Vaz99}.

\paragraph{Context.  Offline vs Real Time Reinforcement Learning} This paper considers  real time policy gradient estimation; that is,  the gradient estimation  and reinforcement  learning is performed in real time
via a stochastic approximation algorithm:
$$ \theta(n+1) = \theta(n) - \epsilon \,\nabla L_n (\theta(n)) $$
where $\theta(n)$ is the parametrized policy at time $n$, $L_n$ is a weak derivative  estimate and $\epsilon$ is a fixed step size algorithm. The fixed step size allows for tracking a time-evolving optimal policy. Note that $L_n(\theta(n))$ is what is termed ``Markovian noise'' \cite{KY03} and we prove weak convergence of the real time policy gradient algorithm to a Kuhn-Tucker point in such a Markovian setting using the ordinary differential  equation approach.

An alternative (and much  simpler case) is to perform  offline policy gradient estimation where the system is run iteratively and each iteration involves generating an entire  independent sample path for a fixed parametrized policy. Such an offline iterative policy gradient algorithm is is of the form:
$$ \theta^{(I+1))} = \theta^{(I)} - \epsilon_I\, \nabla L^{(I)}(\theta^{(I)}) $$
where $I$ denotes iteration number and $\epsilon_I$ is a decreasing step size algorithm.
Such an offline policy gradient converges almost surely to a Kuhn Tucker point
under  much simpler conditions since the individual iterations generate independent sample paths and one does not need to consider the short-term behavior of the gradient estimator. Such decreasing step size  stochastic gradient algorithms with independent noise have been studied since the 1960s.
In contrast,
much of the novelty of the current paper involves constructing and analyzing a weak derivative estimator that does not require knowledge of the transition matrices and operates over a single trajectory of data for real time implementation, and weak convergence analysis in Markovian noise (for a constant step size algorithm).

It is important to note that while in the online case, weak derivatives yield substantially lower variance for the gradient estimate (as shown in this paper), in the offline policy gradient case (since the iterations yield statistically independent trajectories and gradient estimates), the score function gradient estimator and also finite difference methods such as SPSA  \cite{Spa03} work extremely well.

To give further perspective on the  key contributions of this paper, 
 in  Sec.\ref{sec:approaches}, we contrast our approach with 
two other   results  in the  literature, namely, 
 \cite{PNG00} and  \cite{BB99,BB02}
 that also use stochastic gradient
algorithms for implicit adaptive control of MDPs. Also in Sec.\ref{sec:othergrad},
we compare our weak derivative  gradient estimation scheme with two other
widely used gradient estimators in the literature, namely,
perturbation realization factors \cite{Cao98} and the score function estimator
\cite{BB99,BB02}. In Sec.\ref{sec:numerical} we compare the numerical efficiency
of the gradient estimator proposed in this paper with that of the score
function gradient estimator in \cite{BB99,BB02}. We also refer the reader to \cite{HHW06,HV08} for additional background in measure-valued differentiation.

\section{Problem Formulation and Discussion}

In this section we formulate the constrained MDP (\ref{J}),  (\ref{costconstraint}) as a stochastic optimization problem
in terms of parameterized action probabilities. Then for convenience
a summary of the key algorithms in this paper is given. Finally
we briefly summarize how our approach differs from two other approaches
in the literature.

\subsection{Spherically Parameterized Randomized Policies} \label{sec:prr}

The randomized optimal policy for the above constrained MDP can be 
defined in terms of
 the  action probabilities $\th$ parameterized by $\psi$
as:
\begin{eqnarray}
      \prob[u_n = a  | X_n=i] &=& \th_{ia}(\psi), \quad a \in \actionset_i,\;i \in S
   \label{randomized}
\\
  \text{ where } \th_{ia}(\psi) \geq 0,
\quad \sum_{a\in\actionset_i} \th_{ia}(\psi) &=&1 \text{ for every state\ } i
\in S.
\nonumber
\end{eqnarray}
Here $\psi \in \Psi$ is a finite dimensional vector which parameterizes the
action probabilities
 $\th$. $\Psi$ is some suitably defined compact subset of the Euclidean space.
The unconstrained problem with pure optimal policy  is a degenerate case,
where for each $i\in S$, $\th_{ia}=1$ for some $a \in \actionset_i$.

 The most obvious parameterization
$\psi$ for $\th$
-- which we will call {\em canonical coordinates}
 is to choose  $ \psi= \th$. Thus
 $\psi = \{\psi_i\} = 
 \{\th_i\}$, $ i\in S$ is the set of  action 
 probability vectors $(\psi_{ia}; a \in
 \actionset_i)$ satisfying \eq{randomized}.
As discussed in Sec.\ref{sec:approaches}, this 
parameterization has several disadvantages.

In this paper we use a
more  convenient {\em spherical coordinate}
 parameterization $\psi = \a$ that automatically ensures the
feasibility of $\th(\param)$ (i.e., the constraints in (\ref{randomized})
hold) without imposing 
a hard constraint on $\param$.
Such an  approach
was  introduced in \cite{KL96b} for Hidden Markov
Model parameter estimation.
 Adapted to our MDP problem, it reads
as follows: Fix the control agent $i\in S$. Suppose without loss of generality that $\actionset_i = \{0,\ldots,
d(i)\}$. To each value
$\th_{ia}, a\in\actionset_i$ associate the values $\la_{ia} =
\sqrt{\th_{ia}}$. Then \eq{randomized} yields 
$\sum_{a \in \actionset_i} \la^2_{ia} =1$, and
$\la_{ia}$ can be interpreted as the coordinates of a vector that lies
on the surface of the unit sphere in $\Real^{d(i)+1}$, where $d(i)+1$ is
the size of $\actionset_i$ (i.e., number of actions).
  In spherical coordinates, the angles are
$\a_{ip}, p =1,\ldots d(i)$, and the radius is always of size
unity. For  $d(i) \geq 1$, the spherical coordinates
parameterization $\a$ satisfies:
\begin{equation} \label{eq:alfadef}
\th_{ia}(\a) = \la_{ia}^2, \quad
\la_{ia} = 
\begin{cases} \cos(\a_{i,1}) & \text{if } a=0 \\
              \cos(\a_{i,(a+1)}) \prod_{k=1}^a \sin(\a_{i,k}) &
1 \leq a \leq d(i) -1 \\
\sin(\a_{i,d(i))}) \prod_{k=1}^{d(i)-1} \sin(\a_{i,k})  & a = d(i) 
\end{cases}.
\end{equation}
Note that $\th_{ia}(\a) = \la_{ia}^2$ is  an analytic function of $\a$,
i.e., infinitely differentiable in
 $\a$.
It is clear that  under this $\a$
 parameterization,
the control variables $\a_{ip}, p\in\{1,\ldots, d(i)\}$ do not need to
satisfy any  constraints in order
for $\th(\a)$ to be feasible. Furthermore,
 since $\th_{ia}(\a) = \la_{ia}^2$ involves even powers
of $\sin (\a_{i,p})$ and $\cos(\a_{i,p})$,
 it suffices to consider 
 $\a_{ip} \in \Param$ where $\Param$ denotes the compact set
\begin{equation}
 \Param = \left\{ \a_{ip} \in [0,  \pi/2];\; i \in S, \;
p \in \{1,\ldots,d(i)\}\right\}, \label{eq:Al}
\end{equation}
that is, for any $\a\in \Real^{d(i)}$, there is a unique
$\a \in \Param$ which yields the same value of $\th(\a)$.
Let $\Param^o$ denote the interior of the set $\Param^\mu$. Finally,
define the compact set $\Param^\mu\subset \Param^o$ for user
defined small parameter $\mu>0$ as 
\begin{equation}
  \Param^\mu = \left\{ \a_{ip} \in [\mu,  \pi/2-\mu];\; i \in S, \;
p \in \{1,\ldots,d(i)\}\right\}.
\label{eq:amu}
\end{equation}

Notice that $\Param^\mu$ is simply  $\Param$ minus a balls of radius $\mu$ centered
around 0 and $\pi/2$.  As discussed in detail
 after the
 statement of
  Proposition
\ref{prop:SA1} in Sec.\ref{sec:convergence}, excluding $0$ and $\pi/2$ is necessary
to prove convergence of the algorithms we propose -- however, since $\mu$ can be chosen
arbitrarily
small, it is not important in the actual algorithmic implementation.

\subsection{Parameterized Constrained MDP Formulation} \label{sec:paramMDP}
We now formulate the above MDP
problem  as a stochastic  optimization problem
where the instantaneous random
cost is independent of $\param$ but the expectation is with respect to a measure
parameterized by $\param$. Such a ``parameterized integrator''
formulation is common in gradient estimation,
see \cite{Pfl96},
and will be subsequently used to derive  our gradient estimators.
 Consider the augmented (homogeneous) Markov chain
$Z_n \defn (X_n,u_n)$ with state space 
$\Z =S \times \actionset$ and transition probabilities parameterized by $\param$ given by
\begin{eqnarray}
\tp_{i,a,j,a'}(\param)
&\defn& \prob(X_{n+1}=j,u_{n+1}=a'\mid X_n=i,u_n=a) = \th_{j,a'}(\param)\,
A_{ij}(a),
\nonumber\\
&&  i,j \in S,  a \in \actionset_i, a' \in \actionset_j
\label{kernel}
\end{eqnarray}
From the unichain assumption it follows that for any $\param\in {\Param}^o$ 
(interior of  $\Param$), the chain
$\{Z_n\}$ is ergodic, and it possesses a unique invariant probability
measure $\pi_{i,a}(\param); i\in S$, $a \in \actionset_i$.
Let $\Epi$  denote  expectation  w.r.t measure $\pi(\param)$ parameterized by 
$\param$.
From
(\ref{J}) we have
$ J_{x_0}(\policy) = C(\th(\a)) $,
where 
\begin{equation}
C(\th(\a)) \defn \Epi [c(Z)] =
\sum_{i \in S} \sum_{a \in \actionset_i} \pi_{i,a}(\param) c(i,a).
\label{eq:cadef}
\end{equation}
We subsequently denote
$C(\th(\a))$ as $C(\a)$, whenever it is convenient.
The $L$  constraints
\eq{costconstraint} can be expressed as
\begin{equation}
 \G_{l}(\param) \defn \Epi[\con_l(Z)] - \gamma_l
=\sum_{i\in S} \sum_{a \in \actionset_i} \pi_{i,a}(\param)
\con_l(i,a) -\gamma_l \leq 0,
 \quad l= 1,\ldots,L.
\label{B}
\end{equation}
Define $\G = (\G_1,\ldots,\G_l)$.
Thus the optimization problem (\ref{eq:objective1}) with constraints
(\ref{costconstraint}) can be written as
\begin{align}
&\min_{\param\in\Param^\mu} C(\param) 
   \label{optim}
\\
\text{subject to: } & \quad
 \G_l(\param) \leq 0, \quad l=1,\ldots,L .\label{soft}
\end{align}
The constraints $\G_l(\param)$ in (\ref{soft})
will be called
{\bf MDP constraints}. 
 As should be evident from this formulation, the optimal control problem 
(\ref{optim}), (\ref{soft})
depends uniquely on
the {\em invariant} distribution $\pi(\param)$
 of the chain, rather than the (unknown) transition
probabilities
$A_{ij}(u)$.

Recall that expected cost $C(\a)$
and expected constraints
 $\G_l(\a)$ are not known to the decision maker. Our aim is to devise
a recursive (on-line) stochastic approximation algorithm to optimize
(\ref{optim}),  (\ref{soft}) without
explicit knowledge of the transition probabilities $A(a)$.
 Such an algorithm
operates recursively on the observed system trajectory
$(X_{0}, \ldots,X_n, u_0,\ldots,u_{n-1})$  to yield
a sequence of estimates $\{\a(n)\}$ of the optimal solution.
If the unknown 
dynamics  $A(a)$ of the MDP are constant,
then the proposed algorithm ensures that the estimates $\a(n)$
approach the optimal solution. On the other hand, if the unknown 
underlying dynamics  $A(a)$ slowly evolves with time, then our
algorithm will track the optimal trajectory in a sense to be made
clear later.

\begin{remark}
For any  $\param\in {\Param}^o$, the functions $C(\a)$ and $\G_l(\a)$ are 
analytic in $\a$, which follows from the facts that the transition
probabilities in (\ref{kernel}) are linear in $\th(\a)$, and that $\th(\a)$ 
is analytic in $\a$, as mentioned above (\ref{eq:Al}). Furthermore,
due  to the unichain assumption, applying the
continuous mapping theorem to $\pi(\a)$, it follows that $C(\a)$
and $\G_l(\a)$ are continuous in $\Param^\mu$.% see also \cite[Chapter 8.9]{Put94}.
%even if some of the components of the action probabilities $\th(\a)$
%may vanish at the boundary of $\Param$, the stationary probabilities
%$\pi(\param)$ are continuous in $\param$. 
This implies that $C(\a)$
and $\G_l(\a)$ are
uniformly bounded.
\end{remark}

\begin{assumption} 
\label{assum:A2}
The minima
 $\param^*$ of (\ref{optim}), (\ref{soft})
 are regular, i.e., $\nablap \G_l(\param^*)$, $l=1,\ldots,L$ are linearly independent.
Then $\param^*$ belongs to the set of Kuhn Tucker points
\begin{align} \label{eq:KT}
\text{KT} = \biggl\{ \param^* \in \Param^\mu: & \;\exists\, \mu_l \geq 0, \,
l=1,\ldots, L\;\text{ such that
} \quad
  \nablap C + \nablap \G \mu = 0, \quad
 \G^{\p} \mu = 0 \biggr\}
\end{align}
where $\mu = (\mu_1\ldots,\mu_L)^{\p}$.
Moreover $\param^*$ satisfies the second order sufficiency condition
$\nablap^2 C(\param^*) + \nablap^2 \G(\param^*) \mu > 0$ (positive definite)
on the subspace 
$\{ y \in \Real^L: \nablap \G_l(\param^*) y=0\}$ for all $l: \G_l(\param^*)=0,
\mu_l > 0$.
\end{assumption}

\subsection{User's Guide}
 A summary of the key equations for implementing
the learning based algorithm proposed in this paper in spherical coordinates
for constrained MDP (\ref{optim}), (\ref{soft}) is as follows:

{\em Input Parameters}: Cost matrix $(c(i,a))$, constraint matrix $(\con(i,a))$,
batch size $N$.

{\em Step 0. Initialize}: Set $n=0$, initialize $\a(n) \in \Param^o$ and 
vector $\la(n) \in \Real_+^L$. \footnote{More precisely,
$\a(0)$ needs to be initialized in $\Param^\mu$, where
$\Param^\mu \subset \Param^o$ excludes a small $\mu$ size ball around the boundary of 
$\Param$, see 
(\ref{eq:amu}) and Sec.\ref{sec:stochpd}.}

{\em Step 1. System Trajectory Observation}: Observe  MDP over batch
$ I_n \defn \{k \in [nN, (n+1)N-1]\}$ using randomized policy $\th(\a(n))$ of
(\ref{eq:alfadef}) and
compute estimate $\hat{\G}(n)$ of the constraints, (cf.  
\eq{eq:hhat} or (\ref{eq:spd3}) of  Sec.\ref{sec:tradeoff}).

{\em Step 2. Gradient Estimation without explicit knowledge of parameters}: Compute
 $\widehat{\nablal C}(n)$,  $\widehat{\nablal \G}(n)$ over the batch
%$ k \in [nN, (n+1)N-1]$
$I_n$
 using the filtered frozen phantom estimates $\widehat{\widehat{G}}_n$ given in
\eq{eq:gradient}.

{\em Step 3. Update Policy $\th(\a(n))$ using  
constrained  stochastic gradient algorithm}:
Use a penalty function primal dual based stochastic approximation
algorithm  to update $\a$ as follows:
\begin{align}
 \a(n+1) &= \;%\biggl[
\a(n) - \ep \biggl(\widehat{\nablal C}(n) +
\widehat{\nablal \G}(n) \max \left[0,\la(n) + \rho \widehat{\G}(n)
\right]\biggr) %\biggr]\bmod 2\pi
\label{eq:sam1}\\
\la(n+1) &= \max\left[\left(1-\frac{\ep}{\rho}\right) \la(n) ,\la(n) + \ep \widehat{\G}(n) 
\right]. \label{eq:sam2} \end{align}
%where throughout this paper,
%for any vector $\bmod 2 \pi$ is defined element-wise.
The ``penalization'' $\rho$ is a suitably large positive constant and
$\max[\cdot.\cdot]$ above is taken element wise, see (\ref{eq:pd1}),
(\ref{eq:pd2}).

{\em Step 4}. Set $n=n+1$ and go to Step 1.

\medskip

\begin{remark}
1. For large batch size $N$, the bias of the estimates $\hat{\G}$ and 
$\widehat{\nablal \G}$ are $O(1/N)$ and the algorithm 
(\ref{eq:sam1}), (\ref{eq:sam2})
is asymptotically
optimal. However, for  fast tracking of time-varying MDPs it is necessary
to choose $N$ small. In this case, the main source of bias in the estimation
of $\a^*$ using (\ref{eq:sam1}), (\ref{eq:sam2}) is
the covariance of $\widehat{\nablal \G}$ with $\hat{\G}$. Several approaches
to deal with this bias are discussed in
  Sec.\ref{sec:tradeoff}.
\\
2. Another alternative to (\ref{eq:sam2}) is to update $\la$ via a multiplier (augmented
Lagrangian) algorithm (see
Sec.\ref{sec:multiplier}). A third alternative is to fix $\la$. For 
sufficiently large $\rho$, $\a(n)$ will converge
to $\a(\infty)$ which is in a pre-specified ball around
a local minimum $\a^*$.
\\
3. If the true parameters of the MDP jump change at infrequent intervals, then
  iterate averaging \cite{KY03} (as long as the minimal window of
  averaging is smaller than the jump change time) and adaptive step
  size algorithms  can be implemented in the
  above stochastic approximation algorithms to improve efficiency and
  tracking capabilities. 
\end{remark}

\subsection{Discussion of Other approaches in Literature} \label{sec:approaches}
Before presenting the details of the algorithm proposed in this paper,
we briefly summarize two works in the literature
that also use stochastic approximation
methods to solve MDPs.

 The book \cite{PNG00} considers a different parameterization to us --
instead of the parameter $\th$, they consider the parameter to be the invariant measure 
$\pi_{i,a}$ of $Z_n$ described above.
It is well  known \cite{Put94}  that (\ref{optim}) formulated in terms of the invariant measure
$\pi$ is the following
linear program:
\begin{align}
\min_{\pi}  &\sum_{i \in S}\sum_{a\in \actionset_i} \pi_{ia} c(i,a)  \label{linear_prog} \\
\text{ subject to } & \sum_{i \in S}\sum_{a\in \actionset_i} \pi_{ia}\con_l(i,a)   <
\gamma_l \nonumber \\
 \sum_{a \in \actionset_j} \pi_{ja} &= \sum_{i \in S} \sum_{a \in \actionset_j}
 \pi_{ia} A_{ij}(a) ,\quad
 \sum_{i\in S, a \in \actionset_j} \pi_{ia} = 1, \quad 0  \leq \pi_{ia} \leq 1 ,
\quad i \in S , a\in \actionset_j . \nonumber
\end{align}
With $\th^*$ and $\pi^*$ denoting the optimal solutions of (\ref{optim})
and (\ref{linear_prog}), respectively, it is straightforward to show that  
\begin{equation}
\th^*_{ia} = {\pi^*_{ia}\over \sum_{u\in\actionset_i} \pi^*_{iu}} .
\label{eq:theta*}
\end{equation}

The main idea in \cite{PNG00} is to use stochastic approximations for
each component of $\pi_{ia}$ to optimize the above cost function (\ref{linear_prog}).
Because the above cost function is linear in $\pi_{ia}$ the gradient
is merely the observed cost $c(i,a)$. On the other hand, the constraints
in (\ref{linear_prog}) are difficult to handle via stochastic approximation.
 The authors deal with the parameter
constraints using a normalization procedure common in the Learning
Automata Theory. The MDP
constraints are dealt with via Lagrange multipliers, a penalty
function approach and a gradient projection method. The last two
algorithms are direct methods and therefore require estimation of the
transition probabilities $A_{ij}(u)$. Although the Lagrange method does
not require explicit estimation of $A_{ij}(u)$, a closer analysis of the algorithm in \cite{PNG00}
reveals that maximum likelihood estimation is implicitly carried out
in the estimation of the coefficients (gradients) of the Lagrangian
function.

The closest approach to our  paper is that
presented in \cite{BB99,BB02}. The MDP in \cite{BB99,BB02} is without
constraints and uses the parameterization
\[
\th_{ia}(\psi) = {e^{\psi_{ia}}\over\sum_{u\in\actionset_i} e^{\psi_{iu}}},
\quad \psi_{ia} \in \Real, i \in S, \; a \in \actionset_i .
\]
  This exponential parameterization satisfies 
\[
{\partial\th_{iu}  \over \partial \psi_{ia}} = 
\begin{cases}
\th_{iu}(1-\th_{iu}) &  u=a \\
-\th_{iu}\th_{ia} &  u\ne a.
\end{cases}
\]
Using
the chain rule of differentiation on the cost  function
$C(\th(\psi))$ (defined similarly to (\ref{eq:cadef}))
\begin{align}
\part{\psi_{ia}} C[\th(\psi)] &= \sum_{u\in\actionset_i} \part{\th_{ia}} C(\th)\,
\left({\partial\th_{iu}  \over \partial \psi_{ia}}\right) 
= \th_{ia}\left(
\part{\th_{ia}}C(\th) - \sum_{u\in\actionset_i} \th_{iu}\, \part{\th_{ia}}C(\th)
\right), \label{eq:bbp}
\end{align}
which is 
identical to the {\em Generalized Gradient} of \cite{Vaz99}. In our report \cite{GVK03},
we explain why this formulation yields the appropriate descent directional derivative
for canonical coordinates of in Sec.\ref{sec:prr}.
See also \cite{Vaz99} and references therein.
However, gradient algorithms based on this parameterization
 can exhibit slow convergence,
particularly when the optimal probability vector $\th^*$ is
degenerate. When  a component of 
$\th_{ia}$  is zero, (\ref{eq:bbp}) is zero and hence a gradient
based algorithm remains at this point. Because the drift of the update is
proportional to the size of the updated component, as a component
approaches zero, the magnitude of future updates decreases (to prevent
crossing outside the feasible set). This mechanism slows down
convergence of the gradient
algorithm using canonical coordinates $\psi=\th$,
 particularly close to the optimal solution if this has
components representing pure strategies, as is often the case.
We refer the reader to \cite{GVK03} for
numerical examples that demonstrate that the
parameterization involving spherical coordinates has superior
convergence properties compared to canonical coordinates.

In addition, the approach for derivative
estimation in \cite{BB99,BB02} is via the Score Function method, which
usually suffers from unbounded variance for infinite horizon costs. To
alleviate this problem, the authors use a forgetting factor that
introduces a bias in the derivative estimation. Our derivative
estimators are more efficient and consistent, with provably bounded
variance over infinite horizon, in $\Param^{o}$. In Sec.\ref{sec:numerical},
numerical examples show that the variance
of the  score function method is
several orders of magnitude larger than that of the measured valued derivative
estimator.

The above methods all use a ``simulation optimization'' approach,
where almost sure convergence to the true optimal value can be shown
under an appropriate choice of parameters of the algorithms.
In
particular, all stochastic approximations involved in the above
mentioned methodologies use decreasing step size. One of the
motivations of the present work is to implement a stochastic
approximation procedure with constant step size in order for the
controlled Markov chain to be able to deal with tracking slowly
varying external conditions, which result in slowly varying $A(u)$.

\begin{remark}
 Our MDP setting assumes perfect observation of the process
$\{X_n\}$. The paper \cite{BB99,BB02}
considers a partially observed MDP (POMDP) \cite{Kri16}, but assumes
 that the observations $Y_n$ of
the process
$X_n$ belong to a  finite set. In that work they consider suboptimal
 strategies of the form $\th_{ia} = \prob\{u_n=a
\given Y_n = i\}$. Such a policy is clearly
not optimal for a POMDP since the optimal policy is a measurable
function of the history $(Y_1,\ldots,Y_k,u_1,\ldots, u_{k-1})$, which 
is summarized by a continuous-valued information state.
Such  suboptimal POMDP models are  a special case of the problem considered here
and our method can be applied  in a straightforward manner. 
\end{remark}

\section{Deterministic Algorithms for Constrained MDPs}
\label{sec:detalg}

As mentioned above,
 to find the optimal value $\a^*$ defined in (\ref{eq:KT})
(or equivalently, $\th^*$ defined in
(\ref{eq:theta*})), our plan is to use a stochastic approximation
algorithm of the form (\ref{eq:sam1}), (\ref{eq:sam2}). 
A key result in stochastic approximation theory (averaging theory),
see for example \cite{KY03}, states that under suitable regularity and 
stability
conditions, the behavior of the stochastic approximation algorithm is
captured by a {\em deterministic} dynamical system (differential/difference
equation or inclusion) as the step size $\ep$ goes to zero. 
 Thus to design the stochastic approximation algorithms and give
insight into their performance, we will first focus on
designing  deterministic dynamical systems (ODEs) for solving the
constrained MDP \eq{optim}, \eq{soft}. By deterministic we mean
that the objective function and all higher order derivatives
can be exactly computed due to complete knowledge of the parameters of the MDP.
We will construct  suitable ordinary differential equations
(ODEs) whose stable points will be
Kuhn-Tucker points of the optimization problem.

Once the deterministic algorithm 
has been  designed,  the corresponding stochastic
approximation algorithm follows naturally  by replacing the gradient in the
deterministic algorithm with the gradient estimate (which is computed from
the sample path of the Markov chain),
i.e,
by replacing
  $\nablal C$,
$\nablal \G$, $\G$ in the deterministic algorithms presented below
with the estimators
$\widehat{\nablal C}$, $\widehat{\nablal \G}$,
$\widehat{\G}$. These estimates are computed using the parameter free
gradient estimation algorithms given in Sec.\ref{sec:lcmdp}. 
The proofs of convergence of the resulting
stochastic approximation algorithms are given
in Sec.\ref{sec:convergence}.

In this  section, we present a 
 primal dual and an
 augmented Lagrangian algorithm.
Our technical
report \cite{GVK03} also presents a primal
algorithm based on gradient projection which
requires higher computational complexity.

\subsection{First-Order Primal Dual Algorithm} \label{sec:lagrange}
A widely used deterministic optimization method (with extension to
stochastic approximation in \cite[pg.180]{KC78}) for handling
constraints is based on the Lagrange multipliers and uses a
first-order primal dual algorithm \cite[pg 446]{Ber00}. 
First, convert the inequality MDP constraints (\ref{soft}) to equality constraints
by introducing the variables $z = (z_1,\ldots,z_L) \in \Real^L$, so that
$ \G_l(\a) + z_l^2  = 0$, $l=1,\ldots,L$. %Define $\phi \defn (\a,z)$,
%$\G_l(\phi) \defn \G_l(\a) + z_l^2$.
Define the  Lagrangian 
\begin{equation}
\label{Lagrange}
\La(\a,z,\la) \defn C(\a)+ \sum_{l=1}^L \la_l (\G_l(\a) + z_l^2).
\end{equation}
In order to converge, a  primal dual algorithm 
requires the Lagrangian to be locally convex at the the optimum,
i.e., Hessian to be positive definite at the optimum (which is much
more restrictive than the second order sufficiency condition of Assumption
1 in Sec.
\ref{sec:paramMDP}).
  Numerical examples show that this
positive definite condition on the Hessian, which Luenberger 
\cite[pp.397]{Lue84} terms
``local convexity'', seldom holds in the MDP case.
We can ``convexify'' the problem by adding a penalty term to
the objective function (\ref{optim}). The resulting problem is:
$$ \min_{\param\in\Param^\mu,z \in \Real^L} C(\a) 
 + \frac{\rho}{2} \sum_{l=1}^L \left(\G_l(\a) + z_l^2\right)^2, $$
subject to  (\ref{soft}).
Here  $\rho$ denotes a large positive constant.
The optimum of the above problem \cite[pg.429]{Lue84} is identical to that of
(\ref{optim}), (\ref{soft}).
Define the augmented Lagrangian,
\begin{equation} \La_\rho(\param,z,\la) \defn
 C(\a)+ \sum_{l=1}^L \la_l (\G_l(\a) + z_l^2)
 + \frac{\rho}{2} \sum_{l=1}^L \left(\G_l(\param) +z_l^2 \right)^2 .
\label{eq:optimp} \end{equation}
Note that although the original Lagrangian may not be convex near the solution
(and hence the primal dual algorithm does not work),
for sufficiently large $\rho$, the last term in $\La_\rho$ ``convexifies'' the 
Lagrangian. Indeed, for sufficiently large $\rho$, \cite{Ber00} shows that
the augmented Lagrangian is locally convex.
After some further calculations detailed in \cite[pg.396 and 397]{Ber00},
 the primal dual algorithm operating on
$\La_\rho(\param(n),z(n),\la(n))$ reads:
\begin{align}
 \a^{\ep}(n+1) &=\a^{\ep}(n) - \ep \biggl({\nablal C}(\a^{\ep}(n)) +
{\nablal \G}(\a^{\ep}(n)) \max\biggl[0,\la^{\ep}(n) + \rho {\G}(\a^{\ep}(n))
\biggr]\biggr) 
\label{eq:pd1}\\
\la^{\ep}(n+1) &= \max\left[\left(1-\frac{\ep}{\rho}\right)\la^{\ep}(n),\la^{\ep}(n) + \ep {\G}(\a^{\ep}(n))
\right]\label{eq:pd2} \end{align}
where $\ep>0$ denotes the step size and the notation
$z=\max[x,y]$ for any two equal dimensional vectors $x$, $y$
 denotes the vector $z$ with components $z_i = \max[x_i,y_i]$.

\begin{lemma} \label{lem:pdl}
Under Assumption 1,  for sufficiently large $\rho>0$,
there exists $\bar{\ep}>0$, such that for all $\ep \in (0,\bar{\ep}]$,
the sequence $\{\a^{\ep}(n),\la^{\ep}(n)\}$ generated by the primal
dual algorithm (\ref{eq:pd1})
is attracted to a local KT pair $(\a^*, \la^*)$.
\end{lemma}
\begin{proof}  Since $\La_\rho$ is convex for sufficiently
large $\rho>0$ \cite{Lue84}, the proof straightforwardly
follows from 
 Proposition 4.4.2 in \cite{Ber00}.
\end{proof}

Let $T \in \Real^+$ denote a fixed constant and $t\in [0,T]$ denote
the continuous time. Define the piecewise constant interpolated
continuous-time process
\begin{align} \label{eq:ainterpolate}
 \a^{\ep}(t) &= \a^{\ep}(n)  \quad t \in [n\ep, (n+1) \ep) \\
\la^{\ep}(t) &= \la^{\ep}(n) \quad t \in [n\ep, (n+1) \ep).
\label{eq:linterpolate}
\end{align}

Lemma \ref{lem:pdl}
 implies that $\{\la^{\ep}(n)\}$ lies in a compact set $\Lambda$ for all $n$. Note
 $\a^{\ep}(n) \in \Param^\mu$ by definition  where
 $\Param^\mu$ is compact.
Then the following result directly follows from the above
lemma and \cite{KC78} where
the convergence of the  stochastic version is proved.

\begin{theorem} \label{thm:pdconv} Under Assumption 1,
the interpolated process $\{\a^{\ep}(t),
\la^{\ep}(t)\}$ defined in (\ref{eq:ainterpolate}), (\ref{eq:linterpolate})
converges uniformly as $\ep \rightarrow 0$ to the process 
$\{\a(t), \la(t)\}$, i.e., 
\begin{equation}
\lim_{\ep \downarrow 0} \sup_{0 < t \leq T} |\a^{\ep}(t) - \a(t)| = 0 ,
\;\lim_{\ep \downarrow 0} \sup_{0 < t \leq T} |\la^{\ep}(t) - \la(t)| = 0 ,
\label{eq:uniform}
\end{equation}
 where $\{\a(t), \la(t)\}$ satisfy the ODEs
\begin{align} \label{eq:odepd}
 \frac{d}{dt} \a(t)  &= - \nablal C(\a(t)) - \nablal \G(\a(t)) \max\biggl[0,\la(t)
+ \rho \G(\a(t))\biggr]  \nonumber\\
 \frac{d}{dt} \la_l(t) &= 
\begin{cases}\G_l(\a(t))  & \text{ if }\la_l(t) +\rho \G_l(\a(t)) \geq 0\\
               \la_l(t)/\rho    & \text{ if }
\la_l(t) + \rho \G_l(\a(t)) < 0
                              \end{cases}
,\quad l=1,\ldots, L.
\end{align}
The attraction point of this ODE is the local KT pair  $(\a^*, \la^*)$.
\end{theorem}

\subsection{Augmented Lagrangian (Multiplier) Algorithms} \label{sec:multiplier}
We outline two  augmented Lagrangian (multiplier) algorithms. 

\noindent{\bsl 1. Inexact Primal Minimization Multiplier Algorithm:}
The augmented Lagrangian approach
(also known as  a multiplier method) consists of the following coupled 
ODE and difference equation:
\begin{align} \label{eq:diffeqn}
 \frac{d\a^{(n+1)}(t)}{dt} &= -\nablal C(\a^{(n+1)}(t))  - \nablal \G(\a^{(n+1)}(t))
 \max\left[0,\la(n)
+ \rho \G(\a^{(n+1)}(t)\right]  \\
 \la_l(n+1) &= \max\left[0,\;\la_l(n) + \rho \G_l(\a^{(n+1)}(\infty)) \right],
\quad l=1,\ldots,L, \label{eq:multnew}
\end{align}
where $\a^{(n+1)}(\infty)$ denotes the stable point of the ODE (\ref{eq:diffeqn}).
Iteration (\ref{eq:multnew}) is a first order update for the multiplier,
while (\ref{eq:diffeqn}) represents an ODE which
is attracted to the minimum of 
the augmented Lagrangian $\La_\rho$.  The $\max$ in (\ref{eq:multnew})
 arises in dealing
with the inequality constraints, see \cite[pp.396]{Ber00}. 
 \cite[Proposition 4.2.3]{Ber00} shows that if $(\a^{(0)}(0),\la_0(0))$ lies
in the domain of attraction of a local KT pair
$(\a^*,\la^*)$, then (\ref{eq:diffeqn}), (\ref{eq:multnew})
converges to this KT pair.
A practical
alternative to the above exact primal minimization 
is  first order {\em inexact minimization} of the
primal. The iterative
version of the  algorithm reads
\cite[pg.406]{Ber00}: At time $n+1$ set $\a^{(0)}(n+1) = \a(n)$.
Then run $j=0,\ldots,J-1$ iterations of the following gradient minimization of the
primal
\begin{align}
\a^{(j+1)}(n+1) &= \a^{(j)}(n+1) - \ep \biggl( \nablal  C(\a^{(j)}(n)) +
{\nablal \G}(\a^{(j)}(n)) \max\left[0,\la(n) + \rho {\G}(\a^{(j)}(n))\right]\biggr)
 ,
\label{eq:mult1} \\
\intertext{$ \a(n+1) = \a^{(J)}(n+1)$
followed by a first order multiplier step}
\la_l(n+1) &= \max\left[0,\;\la_l(n) + \rho \G_l(\a(n+1)) \right],
\quad l=1,\ldots,L. \label{eq:mult2}
\end{align}
Iteration (\ref{eq:mult1}) represents a first order fixed step size
inexact minimization of the augmented Lagrangian $\La_\rho$ ({\em
  inexact} because (\ref{eq:mult1}) is terminated after a finite
number of steps $J$). It is
shown,  see \cite{Ber00} and references therein,
that as long as the inexact minimization of the
primal is done such that the error tolerances are decreasing with $n$
but summable, then the algorithm converges to a Kuhn Tucker point.
Also  \cite{HSS01} shows that even if the error tolerances are
fixed (i.e., non-decreasing), convergence can be shown for $n
\rightarrow \infty$.

\medskip

\noindent {\bsl 2. Fixed Multiplier}:
 A trivial case of the multiplier algorithm is to fix  $\la(n) =
\bar{\la}$ for all time
$n$ and only
update $\a$ according to (\ref{eq:mult1}) with $I=1$ iteration at
each time instant. This is clearly equivalent to the primal update
(\ref{eq:pd1}) with fixed $\la(n) = \bar{\la}$.
From Theorem \ref{thm:pdconv}, the interpolated trajectory of this
algorithm converges uniformly as $\ep\rightarrow 0$
to the trajectory of the 
ODE
\begin{equation}
 \frac{d}{dt} \a(t)  = - \nablal C(\a(t)) - \nablal \G(\a(t)) \max\left[0,
\bar{\la}
+ \rho \G(\a(t))\right].
\label{eq:fixedmult} 
\end{equation}
The following result in \cite{Ber00}
shows that the attraction point of this   ODE is close
to $\a^*$ for sufficiently large $\rho$, resulting in a near optimal
solution. First convert the $L$ inequality constraints to equality
constraints as outlined in Sec.\ref{sec:lagrange}.
Let $\param^*, \la^*$ denote the corresponding
KT pair.% where $\param^* = (\a^*,z^*)$.

\begin{result} \label{res:mult}\cite[Proposition 4.2.3]{Ber00}.
Let $\bar{\rho} >0$ be scalar such that $\nablap^2 \La_{\rho}(\param^*, {\la}^*) >0$.
Then there exist positive scalars $\delta$ and $K$ such that
for  $(\bar{\la},\rho) \in D \in \Real^{L+1}$ defined by
$$ D = \{(\la,\rho): \|\la - \la^*\| < \delta \rho, \; \rho \geq \bar{\rho}
\}$$
the attraction point $\a^{\la,\rho}$  of the ODE (\ref{eq:fixedmult})
is unique. Moreover, 
$ \|\a^{\la,\rho} - \a^*\| \leq K (\|\bar{\la} - \la^*\|)/\rho$. 
\end{result}

\section{Measure-Valued Gradient Estimation} \label{sec:mvg}
In this section we focus on the derivation of the general formulas for measured-valued gradient estimation
of
 Markov chains. We also
discuss implementation aspects of the ensuing formulas.
In Sec.\ref{sec:lcmdp} we will use these formulas to devise parameter free 
(learning) gradient
estimators  $\widehat{\nablal C}$ and  $\widehat{\nablal \G}$.

 \noindent {\bf Notation}: In this
section we will use 
$\Ep$ to denote expectation w.r.t the underlying probability
 measure of $\{Z_n, n=1,2,\ldots\}$.
 We will also use $\E$ to denote expectation w.r.t. to 
 the distribution of the random action, given that the state is $X_n=i$. 
  certain simulated random variables called {\em phantoms}.
Finally, $Z_n$ will be refereed to as the ``{\em nominal process}.''

\subsection{Measure-Valued Gradient Estimators and Implementation}
 Recall that the transition probability of the chain $\{Z_n\}$ given by
(\ref{kernel}) is parametrized by
$\param$. Denote by $\nablap$ the gradient w.r.t.~the multidimensional parameter
$\param$.
 In \cite{HV2002} it is
shown that the weak derivative of the $n$-th step transition expectation can be
calculated using the chain rule for differentiation, just as in
ordinary calculus; that is, for any test function
%$F:(S\times \actionset)^n\to\Real$,
$F:\Z^n\to\Real$,
\begin{multline}
     \nablap \Ep [ F(\bar Z)] =
             \nablap \left(
               \sum_{\bar i\in S^n} \sum_{\bar u \in \actionset^n} F(\bar i,\bar u)
                \prod_{k=1}^n P_{i_{k-1}, u_{k-1},i_{k},u_k}(\param) 
               \right) \\       
        =     \sum_{k=1}^n\,  \left( 
              \sum_{\bar i\in S^n} \sum_{\bar u \in \actionset^n}
             F(\bar i, \bar u) \, \prod_{l=1}^{k-1} P_{i_{l-1},,u_{l-1},i_{l},u_l}(\param)
  \,
                \nablap P_{i_{k-1},u_{k-1},i_{k},u_k}(\param) \,
                \prod_{l=k+1}^{n} P_{i_{l-1},u_{l-1},i_{l},u_l}(\param)  
                 \right), \label{eq:chain}
\end{multline}
where $\bar Z = (Z_1,\ldots, Z_n)$, $\bar{i} = (i_1,\ldots,i_n)$, $i_k \in S$, 
$\bar{u} = (u_1,\ldots,u_n)$, $u_k \in \actionset_{i_{k}}$,
and each component of $\nablap P(\param)$ is the weak
derivative of the kernel $P(\param)$ w.r.t.~each component of $\param$, as we explain
shortly. While it is a matrix, it does
not  define a transition probability (the rows do not add up to one,  they add up to
zero) so it is not possible to interpret the expression above directly in terms of
``transitions''
to states $i_{k},u_{k}$ starting at $i_{k-1},u_{k-1}$. Using  the concept of weak
derivatives (see \cite{Pfl96}), the transition kernels
$\nablap P_{i_{k-1},u_{k-1}, i_k,u_k}(\param)$ can be interpreted as the weighted
{\em difference} between two
transition probabilities for the random variable $Z_k$, as we now show.

The problem is to find a 
formula for the derivative of the one-step expectation
$\Ep[f(Z_{k+1} ) | Z_k=i_k]$ for {\em any} real valued test function $f \colon
\Z\to\Real$. Using \eq{kernel},  we have 
\begin{equation}
\label{onestepderivative}
\nablap\Ep[f(Z_{k+1}) \given Z_k = (i_k,u_k)]
= \nablap \Ep \left[ \sum_{i\in S}f(i, u_{k+1}) A_{i_k,i}(u_k) 
\given Z_k=(i_k,u_k)
\right]
\end{equation}
where a conditioning argument akin to the method in
\cite{HV2002} has been used to isolate the dependency on $\th(\param)$: given the
state and action pair
$Z_k$ the only dependency on $\param$ is in the distribution of $u_{k+1}$. It
then suffices to evaluate $\nablap \E[ f(i,u)]$ for each fixed value of $i$, with
$\prob[u=a]=\th_{ia}(\param)$. 

In the case of spherical coordinates, the action $u_{k+1}$ (given $i_{k+1}=i$) has a
distribution
\begin{eqnarray}
   u_{k+1} &= & \begin{cases}
            0 & \text{ w.p. } \cos^2 (\a_{i1})\\
             Y_1 & \text{ w.p. } \sin^2 (\a_{i1})
  \end{cases} 
   \label{representation}\\
 Y_1 &= &
\begin{cases}
1 & \text{w.p. } \cos^2(\a_{i2})\\
Y_2 & \text{w.p. } \sin^2(\a_{i2})
\end{cases} 
\nonumber\\
&\vdots & \nonumber\\
Y_{d(i)-1} &= &
\begin{cases}
d(i)-1 & \text{w.p. } \cos^2(\a_{i, d(i)})\\
d(i) & \text{w.p. } \sin^2(\a_{i, d(i)}).
\end{cases} 
\nonumber   
\end{eqnarray}
Let  $Y_{d(i)}=d(i)$.  Because $\a_{ip}$,
$i \in S$, $p \in \{1,2,\ldots d(i)\}$, does not affect
the distribution of $u_{k+1}$ if $i_{k+1}\ne i$, the gradient is non null only
when $i_{k+1}=i$, in which case we have
\begin{eqnarray*}
\part{\a_{ip}} \E [f(i, u)] &=& \ds
\part{\a_{ip}} \E \biggl[
f(i,p-1) \cos^2(\a_{ip}) + f(i, Y_{p})\sin^2(\a_{ip}) 
\biggr] \prod_{k=1}^{p-1} \sin^2(\a_{ik})\\
&=& -2\sin(\a_{ip})\cos(\a_{ip}) 
\prod_{k=1}^{p-1}\sin^2(\a_{ik}) \E[f(i,p-1) - f(i, Y_p)],
\end{eqnarray*}
because the terms $f(i,k), k<p-1$ have weights which are independent of $\a_{ip}$.
The random variable $Y_p$ is called the ``phantom action'' and it has a distribution 
concentrated on $\{p, \ldots, d(i)\}$ corresponding to
\begin{equation}
  \label{Yp}
\prob(Y_p = a) = {\th_{ia}(\param)\over {\ds\prod_{m=1}^{p-1} \sin^2(\a_{im})}},
\quad a \ge p, \;p \in \{1,2,\ldots d(i)\}.
\end{equation}
Notice that by construction,
for $p=d(i)$ the random variable $Y_{d(i)} = d(i)$ is
degenerate. 

The following theorem gives a measure valued gradient estimator
for $\nablal C(\a)$ (cost gradient). 
An identical estimator holds for $\nablal \G_l(\a)$ (constraint gradient) with
$c(Z_n)$ replaced by $\con_l(Z_n)$.

\begin{theorem} 
\label{thm:MVD2} Fix state $i$ and $\a \in \Param^o$. 
Let $\{Z_n=(X_n, u_n)\}$ be an MDP (nominal process)
governed by \eq{randomized} and (\ref{kernel}).
Also for each $k$,  let
$\{Z_n(k) = (X_n(k), u_n(k));  n\ge 0\}$ denote
  a perturbed version (phantom process)
of the MDP $\{Z_n\}$ that follows the
same transition
rules \eq{randomized} and \eq{kernel}, but with initial state 
$$ 
Z_0(k) = 
\begin{cases}
(i, Y_p) & \text{if $p =1 ,\ldots d(i)-1$  where $Y_p$ is randomly generated according to
\eq{Yp}} \\
(i,u_0(k)) & \text{if $p=d(i)$ where }
u_0(k) = \begin{cases}
d(i)-1 &  \text{ if } u_k=d(i)\\
d(i) &  \text{ if } u_k=d(i)-1.
\end{cases} \end{cases}
$$
Then for  $p=1,\ldots, d(i)$, the following measure valued gradient estimator holds:
\begin{equation}
 \part{\a_{ip}}{ \left[{1\over N} \sum_{n=1}^N  c(Z_n) \right]}  =
\frac{2}{N} \sum_{k=1}^N K_{ip}(\a,k) \;
\Ep\left[
\sum_{n=0}^{\min\{\tau(k), (N-k)\}} [c(Z_{n+k})-c(Z_n(k))]
\right].
   \label{MVDGGrad1}
\end{equation}
Here 
$
\tau(k) = \min\{n>0  : Z_n(k) = Z_{n+k}\}
$  denotes the coupling time, and
\begin{equation} \label{eq:kip}
 K_{ip}(\a,k) = 
\begin{cases}
 -\tan(\a_{ip}) \,\delta_{i,p}(k) & \text{ if } p=1,\ldots,d(i)-1 \\
- \cos(\a_{i,d(i)}) \sin(\a_{i,d(i)})
[\delta_{i,d(i)}(k) - \delta_{i,d(i)+1}(k)] &
\text{ if } p =d(i) \end{cases}
\end{equation}
where $\delta_{ip}(k) = \ind{Z_k =(i,p-1)}$.
\end{theorem}

\begin{remark}
1. Evaluating the phantom processes 
$k$ only for those steps where 
$u_k=p-1$ in the nominal path significantly saves computational effort
and as  will be
discussed in
Sec.\ref{sec:gradients},
forms the basis for parameter-free
gradient  estimation. For those steps, the initial state
of the plus system will
have the same decision as the observed one, that is, 
$u_k^+(k)=u_k=p-1$ and only one other random variable $u^-_k(k)=Y_p$ has to be
simulated.  \\
2. If $\a_{ip}=\pi/2$, then $\th_{i,p-1}=0$ so that $\delta_{i,p}=0$ and
we set $K_{ip}(\pi/2,k) = 0$ in (\ref{MVDGGrad1}).
\end{remark}

\begin{proof}
Consider $p<d(i)$ and call $F(\bar Z)$ the sample average cost $(1/N)\sum_n  c(Z_n)$.
It follows from the chain rule and the development of the one-step transition derivative
kernel that:
\[
 \part{\a_{ip}} \left(\Ep[F(\bar Z)]\right) =  
-2\sin(\a_{ip})\cos(\a_{ip}) 
\prod_{m=1}^{p-1}\sin^2(\a_{im})
\Ep \biggl[\sum_{k=1}^N\, \ind{X_k=i}\, \E[ F(\bar Z^+(k)) -
F(\bar Z^-(k)) ]\biggr],
\]
where for each $k$, $\{Z_n^\pm(k), n\le k\}$ is a Markov process with
transition matrix $\tp(\th)$. Next, the ``plus'' and ``minus'' processes have actions
$u_k^+ = p-1$ and $u_k^-=Y_p$ (with distribution as in \eq{Yp}). Then the evolution
of the processes follows: $\prob[Z^\pm_{k+1}(k)=(j,a')|Z_k^\pm(k)=(i,Y^\pm)]=
A_{ij}(Y^\pm) \th_{ja'} $, and $\{Z_n(k), n> k\}$ again is a Markov
process with transition matrix $\tp(\th)$. 
For each $k$, an instance
of the paths up to step $k$ is the nominal process itself, therefore
choosing $Z^\pm_n(k) = Z_n, n\le k$ will yield the same expectation
for the gradient, and the first terms in the difference of
sample averages cancel out.

Equivalently,
the plus and minus processes can be stated as MDP's where the decision at step $k$ is
``forced'' to have the values
$p-1, Y_p$ respectively, that is, for each $k$ the MDP $Z_n^\pm(k) =
(X_n^\pm(k), u_n^\pm(k))$ evolves according to \eq{randomized} and \eq{kernel}. 
Using this particular representation, the trajectories ``split'' the decisions: the {\em
nominal} decision is the one observed:
$u_k\sim \th_i(\param)$, the decision in the ``plus'' system is $u_k(k)^+=p-1$ and the
decision in the ``minus'' system is distributed according to $ \tilde\th^{(a)}_i$.  From
there on, all processes follow the same dynamics for the MDP, namely equations
\eq{randomized} and \eq{Aij}.

Consider the case  $p<d(i)$.
Sample  the nominal process to obtain an instance of the ``plus'' process, that is,
whenever $u_k=p-1$ we consider the nominal process as the ``plus'' process.
Because   this observation has a sampling rate of $\th_{i,p-1}(\param) =
\cos^2(\a_{ip})\prod_{m=1}^{p-1}\sin^2(\a_{im})$, then 
\begin{align*}
 \part{\a_{ip}}\left(\Ep[F(\bar Z)]\right) = & 
-{2\tan(\a_{ip}) \over N}
\Ep \biggl[\sum_{k=1}^N\, \ind{Z_k=(i,p-1)} %\times\\
  \E\left[
\sum_{n=k}^N  [c(Z^+_n(k)) - c(Z^-_n(k))] \right] \biggr].
\end{align*}
The first line of (\ref{eq:kip}) holds by
identifying $X_n(k) = X^-_{k+n}(k), u_n(k) = u^-(k)_{n+k}; n\ge 0$, and using the fact that
for each $k$, after the coupling time $\tau(k)$ both MDPs have the same
distribution. 

Now consider the case $p=d(i)$. Since $Y_p$ is degenerate, we can sample
the nominal process when it has decision $d(i)-1$ or $d(i)$. The event
$u(k) \geq d(i)-1$ happens with probability $\prod_{m=1}^{d(i)-1} \sin^2(\a_{im})$.
If the decision in the nominal is $d(i)$, then the contribution to the 
derivative is negative. Otherwise it is positive. Hence the second line of
(\ref{eq:kip}) holds.
\end{proof}

\subsection{Comparison with other Gradient Estimation Methods}
\label{sec:othergrad}
The aim here is to briefly compare the gradient
estimation formula  (\ref{MVDGGrad1}) in  \thm{MVD2} with two other widely used
gradient estimators in the literature, namely, realization
perturbation factors and the score function estimator.

{\bf Realization Perturbation Factors}:
The realization perturbation factors of \cite{Cao98},
can be used to estimate the difference between the steady state costs
for two different transition probability matrices.
This has been used in \cite{Cao98} to derive a 
simulation based policy iteration algorithm for an unconstrained MDP.
 To relate our formula (\ref{MVDGGrad1}) to the realization perturbation 
formulas
in \cite{Cao98}
 use the following argument.
In the long run, as $N \rightarrow \infty$,
the fraction of steps where $\delta_{ia}(k)=1$ is
$\pi_{i,p-1}(\a)$
% $\th_{ia}\,\pi_i(\th)$,
that is,
the stationary probability of the chain $Z_n=(X_n,u_n)$. 
With $\tilde{u} \sim Y_p$
satisfying (\ref{Yp})
 and $\tau$ denoting  the coupling time of the two
Markov chains, it can
 be shown that
 \begin{align}
\frac{\partial}{\partial \a_{ip}} \Epi[c(Z)] = & 
-2 \tan(\a_{ip}) \pi_{i,p-1}(\a)\, %\times
%\\
%&
 \E \left[ 
\sum_{n=0}^{\tau}  \Ep[c(Z_n)\given Z_0=(i,a)] - \Ep[c(Z_n)\given Z_0=(i,\tilde u)]
\right] \nonumber\\
&\hspace{-3cm}= -2 \frac{\tan(\a_{ip})}{\prod_{m=1}^{p-1} \sin^2(\a_{im})}
\pi_{i,p-1}(\a) 
\sum_{u=p}^{d(i)} \th_{i,p-1}(\a)\, %\times
%\\
%&
 \E \left[ 
\sum_{n=0}^{\tilde \tau(u)}  \Ep[c(Z_n)\given Z_0=(i,a)] - \Ep[c(Z_n)\given Z_0=(i,
u)]
\right]\label{eqivrp}
\end{align}
for $p=1,\ldots, d(i)-1$, where the last line of
the above equation follows from
using a conditioning argument on the values of $u$.
In (\ref{eqivrp}),
 (abusing notation) $\tilde \tau(u)$ is now the corresponding coupling time of the
processes started at $(i,a)$ and $(i,u)$. Eq.(\ref{eqivrp})
 is the 
spherical coordinate equivalent of weights of  the
perturbation realization factors in terms of stationary probabilities, in \cite{Cao98} and
\cite{Dai00}, see also \cite[Lemma 3.75, pp.203]{Pfl96}.

\medskip

{\bf  Score Function Estimator of Bartlett \& Baxter \cite{BB99,BB02}}:
The Score Function gradient estimator 
of (\ref{eq:bbp})
can  be derived using the measure-valued
approach. Let an arbitrary finite-valued
random variable $Z$ take value $j$ with probability
$p_\th(j)$. Then  
\[
\part{\th_{ia}} F(Z) = \part{\th_{ia} } \sum_j F(j)\, p_\th(j) = \sum_j
\left( \part{\th_{ia}} \ln [p_\th(j)] \right) F(j) p_\th(j) = \Ep[F(Z)
S(\th_{ia},Z)]
\]
where $S(\th_{ia},Z)  \defn  \part{\th_{ia}} \ln [p_\th(Z)]$ is known as the Score Function. 
Returning to the Markov process
$\{Z_n\}$, 
consider the estimation of 
 \eq{onestepderivative}. Then   using $P(Z_{k+1}=(i,a)|Z_k=(j,u)) = 
A_{ji}(u) \th_{ia}$ yields
\[
S(\th_{ia}, Z_{k+1}) = 
\frac{1}{\th_{ia}} \ind{Z_{k+1} = (i,a)},
\]
which is not uniformly bounded in $\th$: when one
or more components of the control parameter tend to zero (which they
do when a policy is pure instead of randomized) the estimator blows
up. To overcome this problem a Score Function estimator is used in
\cite{BB99,BB02} with the exponential parameterization.
 When inserted in the formula (\ref{eq:bbp}) for the chain rule,
the Score Function estimator for the Markov Chain is of the form
$\sum_n S(\th_{ia}, Z_n)$ and it is a well known problem that the
variance increases with time. Numerous variance reduction techniques
have been proposed in the literature \cite{Pfl96} including
regenerative estimation, finite horizon approximations and \cite{BB99,BB02}  propose to use a forgetting factor for
the derivative estimator. Their method suffers therefore of a
variance/bias trade-off, while our estimation method is consistent and
has uniformly bounded variance (in $N$), as will be shown shortly.

\section{Gradient Estimation without explicit knowledge
of parameter values} \label{sec:lcmdp}
In this section we show how to modify the gradient estimation algorithms
of Theorem \ref{thm:MVD2}
to make them parameter free. That is, the algorithms presented
below are simulation based and do not require explicit knowledge of the
transition probabilities of the constrained MDP. The algorithms
use a novel concept called {\em frozen phantoms.}

\noindent {\bf Remark. Doeblin phantoms vs Frozen phantoms}:
Typically in the
Discrete Event Systems literature (see \cite{Cao98,Dai00} and references therein),
it is assumed that simulations can 
 be performed off-line.
Extrapolating this philosophy to the phantom processes $\{Z_n(k)\}$, $k=1,2\ldots,N$
 defined in Theorem \ref{thm:MVD2}, an obvious first attempt
would be to assume that these phantoms can be  simulated off-line.
Following the terminology of Doeblin simulations in \cite{Dai00}, which consists
of independent, off-line simulations, 
we define a {\em Doeblin phantom} process $\{Z_n(k)\}$, $k=1,2\ldots,N$ as follows:
Start  with an observed state $Z_{n+k}=(i,a)$ and let the
 the phantom decision be
$u_0(k)$. Then generate $Z_n(k)$ 
 independently of the observed process $\{Z_{n+k}\}$ 
for $n=0,\ldots,
\tau(k)$
 with
 identical Markovian dynamics.
Due to the coupling property of Markov chains, the Doeblin phantom
and the nominal  process merge at finite (a.s.) time $\tau(k)$.
For $n>\tau(k)$, the Doeblin phantom process is defined to be identical
to the  observed process.
%Fig.\fig{doeblin} depicts a Doeblin phantom.
%
%\begin{figure}[h] 
%\centering
%\setlength{\epsfxsize}{3in}
%\epsfbox{doeblin.eps}
%\caption{Doeblin phantoms evolve in parallel until coupling with nominal}
%\label{fig:doeblin}
%\end{figure}

Unfortunately, despite the widespread usage of Doeblin simulations, 
Doeblin phantoms are not suitable for our learning problem
since they require explicit knowledge of the transition probability matrix $A(u)$.
When
$A(u)$ is not known, it is of course still possible to try to
estimate it  concurrently with the gradients.
Instead, we propose a new method, called {\em frozen phantoms}
that overcomes this difficulty and
gives the basis for indirect adaptive  control of the constrained MDP.
We also present a short-batch version called {\em fast frozen phantoms} that
can be used in adaptively controlling time varying MDPs.

\subsection{Frozen Phantoms for Gradient Estimation without explicit
knowledge of parameters}
\label{sec:gradients}
If a phantom system
starts at state $i$ and a given decision
$\tilde u\in\actionset_i\setminus \{a\}$,  the history of the process may be used 
as a stochastic version of
this system: for example one can wait until the nominal process  has state-decision pair
$(i,\tilde u)$. 
From then on the cost of the phantom system can be observed from the nominal path,
without the need for simulations.  

Our methodology is as follows:
First, a cut-and-paste argument (see \cite{HC91}) is used. The phantom
system is `frozen' at the initial state for $\nu$ iterations, until the
nominal system hits this phantom state (which happens in finite
time a.s.\ for $\a \in \Param^o$).
 Once the systems couple they follow identical paths until the
nominal system has completed $N$ stages, at which point the phantom
system must complete the $\nu$ remaining steps. The novel idea that we
introduce is to filter these dynamics instead of simulating the remaining
$\nu$ steps of the phantom systems. Filtering the phantoms implies
averaging out the dynamics of the phantom system so no further detailed
simulation is required. Apart from making our gradient
estimation parameter free,
it is  more efficient than  off-line simulations
because no extra CPU time is required.

The theorem below presents a frozen phantom  gradient estimator
version of  (\ref{MVDGGrad1})
and  does not require explicit knowledge of the transition probabilities
of the MDP.  A similar gradient 
 estimator holds for the constraints.

\begin{theorem} \label{thm:frozen}
Consider the MDP  $\{Z_n\}$ governed by
 \eq{randomized} and \eq{kernel} with $\param \in \Param^\mu$.
 Fix $(i,p)$, $a=p-1$. 
 Let $\{\tilde u(k)\}$ be a sequence of iid random variables with distribution 
%$\th^{(a)}_i$ in (\ref{eq:tha}) or 
$Y_p$ in (\ref{representation}) for $p<d(i)$, and for $p=d(i)$, let
$\tilde u_{k} = d(i)\ind{u_{k}=d(i)-1}+(d(i)-1)\ind{u_{k}=d(i)}$,
 independent of
$\Field_n$. Define the hitting time
\begin{equation}
\nu(k) =  \min\{n>  0  : Z_{k+n} = (i, \tilde u(k))\}.
\label{eq:nudef} \end{equation}
Let $ \hat{\Gamma}_N $ denote any estimator that converges
as $N\rightarrow \infty$  to $C(\a)$ a.s.\ (e.g.,
$ \hat{\Gamma}_N = \frac{1}{N} \sum_{n=1}^N  c (Z_n)$) and $K_{ip}(\a,k)$ defined in (\ref{eq:kip}). 
Then for fixed $\a \in \Param^o$, a
parameter free consistent estimator for the gradient in (\ref{MVDGGrad1}) 
is given by:
\begin{align}
\hat G_N (i,p) &=
 \frac{2 }{ N }
 \sum_{k=1}^N K_{ip}(\a,k)
\left( 
[c(i,u_k)-c(i,\tilde u(k))]
+ \sum_{n=k+1}^{k+ \nu(k)}  c(Z_n) - \nu(k)\,   \hat{\Gamma}_N 
\right), \quad p=1,\ldots,d(i).
   \label{GGradEstimator2}
\end{align}
\end{theorem}

\begin{proof}
Consider the homogeneous Markov
chain $\{Z_n\}$. Because for $\param \in \Param^\mu$ the chain is aperiodic and
irreducible in a finite state space, it is geometrically ergodic with
a unique stationary distribution. Hence there
is a
constant  $0<\rho<1$ 
so that for any function $d: S \to\Real$, there is a positive constant
$K_d<\infty$
such that
\[
| \esp[d(Z_n) \given X_0] - \esp[d(Z_\infty) | \le K_d\rho^n \text{ a.s.},
\]
where $Z_\infty$ has the stationary distribution of the chain. This in turn implies that for
each $k$, the sum of the difference processes is absolutely summable: 
\begin{equation}
\sum_{n=1}^\infty | \esp[d(Z_n)\given Z_0=(i,p-1)] - \esp[d(Z_n) \given Z_0= (i,\tilde u)]  |
\le
2 K_d \sum_{n\ge 0}\rho^n < \infty. \label{eq:summable}
\end{equation}
For notational convenience, since for fixed $\param \in \Param^o$,
$K_{ip}(\a,k)$  in (\ref{MVDGGrad1}) is bounded, it suffices to work with
\begin{equation}
 L\defn  \lim_{N\to\infty} \frac{1}{N}\sum_{k=1}^N \delta_{ip}(k)\,\times
\Ep
\left[
\sum_{n=k}^{N} [c(Z_{n})-c(Z_n(k))]
\right] \label{eq:LN}
\end{equation}
instead of  (\ref{MVDGGrad1}).
Then (\ref{eq:summable})  together with the dominated convergence 
theorem can be used to interchange limits and expectations in
(\ref{eq:LN}) and establish that:
\[
\esp[L]
=
\lim_{N\to\infty} 
\frac{1}{N} 
\sum_{k=1}^N \esp\left( \delta_{ip}(k)
\sum_{n=k}^N (c(Z_n) - [c(Z_n(k)])
\right)
\] 
where $Z_n(k)=(X_n(k), u_n(k))$ is started at the point $(i,\tilde u(k))$ which follows
\eq{randomized} and \eq{Aij}. Construct a version of this process via the nominal 
process $\{Z_n\}$:
\[
Z_n(k) = \begin{cases}
(X_n, u_n) & n < k \\
(i, \tilde u) & n=k \ (\text{ if } X_n = i)\\
(X_{n+\nu(k)}, u_{n+\nu(k)}) & n>k .
\end{cases}
\]
The idea of the hitting
time  until the nominal reaches a phantom system is illustrated in 
Fig.\ref{fig:frozen}.

\begin{figure}[h] 
\centering
\setlength{\epsfxsize}{3in}
\epsfbox{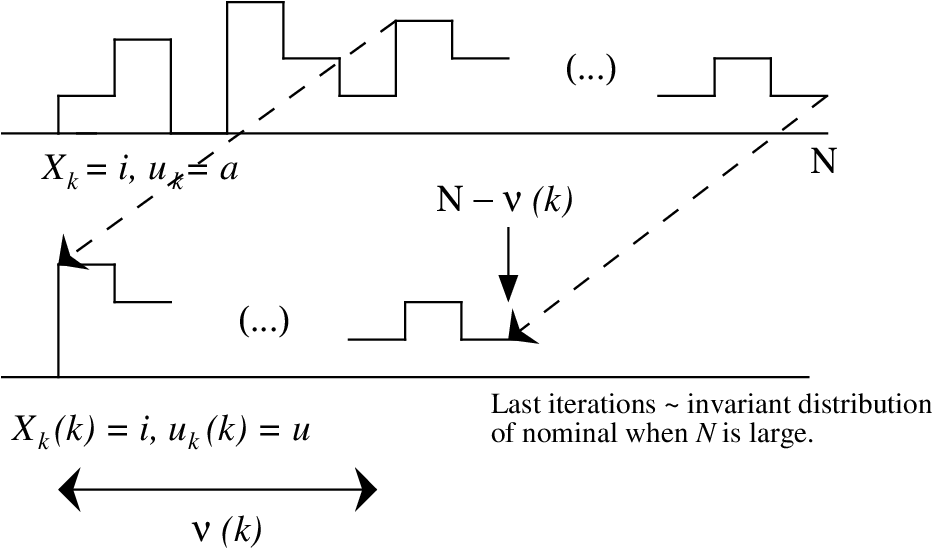}
\caption{Frozen phantoms wait until nominal hits the initial state.}
\label{fig:frozen}
\end{figure}

Using this version of the process, for each $k$, the difference of the finite horizon sum in
the inner brackets is:
\begin{align*}
& c(i,p-1) - c(i,\tilde u(k)) + 
\sum_{n=k+1}^N c(Z_n) - \sum_{n=k+ 1+\nu(k) }^{N + \nu(k)} c(Z_n)\\
&=c(i,p-1) - c(i,\tilde u(k)) + 
\sum_{n=k+1}^{\min\{(k+\nu(k)),N\}} c(Z_n) - \ind{k+\nu(k)<N}
\sum_{n=N+1}^{N + \nu(k)} c(Z_n).\\
\end{align*}
We now show that 
\begin{equation}
\label{bound}
\lim_{N\to\infty} {1\over N} \sum_{k=1}^N \delta_{ip}(k)
\esp \left( \ind{k+\nu(k) >N}\sum_{n=N+1}^{N+\nu(k)} | c(Z_n) |
\right) = 0.
\end{equation}

First, notice that $|c(Z_n)|<K_1$ is uniformly bounded because the
state space is finite. 
Define $\tau(u) = \inf\{k\geq 1, \;Z_k=(i,u)\}$ as the first return
time to state $(i,u)$.
Clearly for each value of $ u\in
\actionset_i$, the hitting time until the first return to $(i,u)$
(starting from $(i,a)$) is bounded a.s.~by $\tau(u)$.
Because $\nu(k)$ is a
first hitting time for some $\tilde u(k) \in \actionset_i$, then
since $\a \in \Param^o$, it follows that
$\nu(k)\le \tau$ a.s., where $\tau = \max(\tau(u) \colon
u\in\actionset_i)$ is finite a.s. 
This implies that for any $N$, the sum in
\eq{bound} is bounded by: 
\[
K_1\, \esp \frac{\tau}{N}\left( 
\sum_{k=1}^N \delta_{ip}(k) 
\ind{\tau> N-k }
\right).
\]
The Markov chain
$\{Z_n\}$ is positive recurrent, implying that $\tau$ is a.s.~finite and
$\esp\tau = 1/\pi_{iu}(\alpha))< \infty$ for some $u\in\actionset_i$. Therefore for any
$\ep>0$ there exists $b_\ep<\infty$ (independent of $N$) such that $\esp[\tau \ind{\tau>b}]
<\ep$. Take $N> b_\ep$ and calculate
\begin{align*}
& K_1\, \esp {\tau\over N} \left( 
\sum_{k=1}^N \delta_{ip}(k) 
\ind{\tau> N-k }
\right) \\
& \le 
K_1\, \esp {\tau\over N}\left( 
\sum_{k=1}^{N-b_\ep} \ind{\tau > N-k} 
+ \sum_{k=N-b_\ep+1}^{N}\ind{\tau> N-k }
\right)\\
& \le 
K_1\, \esp {\tau\over N}\left( 
\sum_{k=1}^{N-b_\ep} \ind{\tau > b_\ep} 
+ \sum_{k=0}^{b_\ep}\ind{\tau> k }
\right)\\
&\le K_1\, \left(\ep +{b_\ep \esp\tau\over N}\right),
\end{align*}
which shows \eq{bound}. With \eq{bound}, it follows that
\[
\esp[L] = \lim_{N\to\infty} {1\over N} \esp\left[
\sum_{k=1}^N \delta_{ip}(k) 
 \left(
c(i,u_k) - c(i,\tilde u(k) ) + 
\sum_{n=k+1}^{k+\nu(k)} c(Z_n) -  
\sum_{n=N+1}^{N + \nu(k)} c(Z_n)
\right)\right].
\]
 
By the Markov property, the last sum above can also be expressed by $\sum_{m=0}^\infty
c(Z^{(N)}_m)
\ind{m<\nu(k)}$, where $Z^{(N)}_0\sim \PP^N$ has the distribution of the $N$-step
transition of the nominal chain. This distribution converges to the invariant distribution as
$N$
grows.
Because $ \{Z_k\}$ is time  homegeneous,
the distribution of  $\nu(k)$ (see (\ref{eq:nudef}))
is independent of $k$. Therefore,
\[
\lim_{N\to\infty} \Ep \biggl[\sum_{m=0}^\infty c(Z^{(N)}_m) \ind{m<\nu(k)} 
\biggr]
=
\Ep \biggl[\sum_{m=0}^\infty C(\a) \ind{m<\nu(k)}\biggr] = C(\a) \Ep[\nu(k)],
\] 
which completes the proof, because $\nu(k)$ are  
uniformly bounded in $N$ and $ \hat{\Gamma}_N  \to C(\a)$ a.s. so the expectation of the
product will converge to  $C(\a) \Ep[\nu(k)]$.
\end{proof}

\subsection{Numerical Comparison of Efficiency of Gradient Estimators}
\label{sec:numerical}

\noindent {\bf System Parameters}:
 We simulated the
following MDP: $S=\{1,2\}$ (2 states), $d(i)+1=3$ (3 actions),
\[
A(0) = 
\begin{pmatrix}
0.9 & 0.1 \\
0.2 & 0.8 
\end{pmatrix}, \;\;
A(1) =
\begin{pmatrix}
0.3 &  0.7\cr
0.6 & 0.4 \cr
\end{pmatrix}, \;\;A(2) = 
\begin{pmatrix}
0.5 &  0.5\cr
0.1 & 0.9 \cr
\end{pmatrix}.
\]
The action probability matrix  $(\th(i,a))$ and cost 
matrix $(c(i,a))$ were chosen as:
\[  ( \th(i,a)) = \begin{bmatrix}
0.2 &  0.6 &  0.2 \\ 0.4&  0.4&   0.2 \end{bmatrix},\quad
(c(i,a)) = 
-\begin{bmatrix}
50.0 &   200.0 &  10.0 \\ 
3.0 &    500.0 &   0.0 \end{bmatrix}
\]

\noindent {\bf Gradient Estimates in Spherical Coordinates}:
The theoretical values of the 
gradients are 
$$ \nabla_\a C_N(\a) =
\begin{pmatrix}
45.05 &  -55.07 \\
187.58   & -159.91    
\end{pmatrix} $$
The gradient estimates for  batch sizes $N=100, 1000$ are: 
\begin{align*}
\widehat{\nablal C}^{\text{Doeblin}}_{100} &=
\begin{pmatrix}
41.720 \pm 4.695    &  -55.096 \pm 1.464 \\
197.166 \pm 8.062   &  -164.701 \pm 4.610
\end{pmatrix}\\
\widehat{\nablal C}^{\text{Doeblin}}_{1000} &= 
\begin{pmatrix}
  41.703 \pm 1.278   &  -53.667 \pm 0.455 \\
191.249 \pm 2.858   &  -167.048 \pm 1.228
\end{pmatrix}\\
 \widehat{\nablal C}^{\text{Frozen}}_{100} &=
\begin{pmatrix}
43.333 \pm 4.791    &  -54.191  \pm 2.096 \\      
196.956 \pm 8.951   &  -161.200 \pm 4.918
\end{pmatrix}\\
\widehat{\nablal C}^{\text{Frozen}}_{1000} &= 
\begin{pmatrix}
44.322 \pm 1.323    &  -53.656 \pm 0.712\\
189.549\pm 2.829  &  -164.329 \pm 1.231
\end{pmatrix}
\end{align*}
In the above expression, the numbers following the $\pm$ sign
are confidence
intervals which were estimated at level
$0.05$ using the normal approximation with $100$ batches.

\begin{table}[h]
\centering
\begin{tabular}{|c| c | c  || c | c |  } \hline
$N=1000$  & \multicolumn{2}{c|}{$\var [\widehat{\nablal
C}^{\text{Doeblin}}_N]$}  & 
  \multicolumn{2}{|c|}{$\var [\widehat{\nablal C}^{\text{Frozen}}_N]$} \\
 \hline
$ i = 0 $ & 42.558 &   5.404  &  45.604 &  13.206 \\
 $i = 1 $ & 212.74 & 39.26   &  208.43 & 39.431  \\
\hline
CPU & \multicolumn{2}{|c|}{4 secs.}  & \multicolumn{2}{c|}{2 secs.} \\
\hline
\end{tabular}
\caption{Variance of Doeblin phantom vs Frozen phantom, spherical coordinates}
\label{tab:PhanAlpha}
\end{table}

The variance matrix is in \tab{PhanAlpha}.  It is clear from the
numerical experiments that the frozen phantom implementation not only is robust
(thus more appropriate when the underlying parameters are unknown) but
also more efficient; the CPU time is about half of that using Doeblin
phantoms, yet their variances are comparable.

\noindent{\bf  Frozen Phantoms versus Score Function Method}:
As mentioned in Sec.\ref{sec:approaches} the closest approach to the algorithms
in this paper is that in \cite{BB99,BB02}, which uses a Score Function method to
estimate the gradients. Here we compare our frozen phantom estimator
with the score function gradient estimator of \cite{BB99,BB02}.
 Since the score function
gradient estimator  in \cite{BB99,BB02} uses canonical coordinates $\th$,
to make a fair comparison
in this example we work with canonical coordinates.
The theoretical values of the generalized gradient (\ref{eq:bbp})
for the above MDP are
\begin{equation} \label{eq:theorval}
\nabla_{\psi}[C(\th(\psi))]
=\begin{pmatrix}
  -9.010  & 18.680 &   -9.670  \\    
-45.947  & 68.323 & -22.377            
\end{pmatrix}. 
\end{equation}
We simulated the frozen phantom  and
score function estimators for (\ref{eq:bbp}) in canonical coordinates,
see \cite{GVK03} for 
 implementation details.
 For 
batch sizes $N=100$ and $1000$,  the 
frozen phantom gradient estimates are
% (to be 
%compared with the theoretical values in (\ref{eq:theorval})):
\begin{align*}
 \GF_{100} &=
\begin{pmatrix}
-7.851 \pm 0.618    &  17.275 \pm 0.664   &  -9.425 \pm 0.594 \\
-44.586 \pm 1.661 &  66.751 \pm 1.657    &  -22.164\pm 1.732 \\
\end{pmatrix}\\
\GF_{1000} &= 
\begin{pmatrix}
 -8.361 \pm 0.215    &  17.928 \pm 0.240    &  -9.566 \pm 0.211 \\
 -46.164 \pm 0.468   &  68.969 \pm 0.472    &  -22.805 \pm 0.539 \\
\end{pmatrix}.
\end{align*}
Again the numbers after $\pm$ above, denote the confidence
intervals at level
$0.05$  with $100$ batches.
The variance of the frozen phantom gradient estimator  is shown in
\tab{PhanTheta}, together with the corresponding CPU time. 
\begin{table}[h]
\centering
\begin{tabular}{|c|| c | c | c| } \hline
$N=1000$  &  
  \multicolumn{3}{|c|}{$\var [\GF_N]$} \\
 \hline
$ i = 0 $ & 1.180 &  1.506 & 1.159   \\
 $i = 1 $ &  5.700 & 5.800 &  7.565 \\
\hline
CPU   & \multicolumn{3}{c|}{2 secs.} \\
\hline
\end{tabular}
\caption{Variance of Frozen phantom in  canonical coordinates}
\label{tab:PhanTheta}
\end{table}

We implemented the score function gradient estimator of \cite{BB99,BB02}
with the  following
parameters: forgetting
factor $1$ (otherwise the estimates are biased),  batch
sizes of $N=1000$ and 10000. In both cases a total number of $10,000$
batches were simulated.
% which required CPU times of 1348 seconds and
%13492 seconds, respectively. 
The score function gradient estimates are
\begin{align*}
 \GS_{10000} &=
\begin{pmatrix}
 -3.49  \pm 5.83     &  16.91 \pm 7.17 & -13.42 \pm 5.83 \\
 -41.20 \pm 14.96 &  53.24 \pm 15.0  &  -12.12 \pm 12.24 
\end{pmatrix}\\
 \GS_{1000} &= 
\begin{pmatrix}
-6.73 \pm 1.84 & 19.67 \pm 2.26 &  -12.93 \pm 1.85 \\
-31.49 \pm 4.77 & 46.05 \pm 4.75 &  -14.55 \pm 3.88
\end{pmatrix}
\end{align*}
The variance of the score function gradient estimates are given
Table \ref{tab:PhanScore}. 
\begin{table}[h]
\centering
\begin{tabular}{|c| c | c | c |} \hline
$N=1000$  & \multicolumn{3}{c|}{$\var [\GS_N]$}  \\
 \hline
$ i = 0 $ & 89083 & 135860 & 89500  \\
 $i = 1 $ & 584012 & 593443 & 393015 \\
\hline
CPU & \multicolumn{3}{|c|}{1374 secs.}  \\
\hline
\end{tabular}
\begin{tabular}{|c| c | c | c |} \hline
$N=10000$  & \multicolumn{3}{c|}{$\var [\GS_N]$}  \\
 \hline
$ i = 0 $ & 876523 & 1310900 & 880255  \\
 $i = 1 $ & 5841196 & 5906325 & 3882805 \\
\hline
CPU & \multicolumn{3}{|c|}{13492  secs.}  \\
\hline
\end{tabular}
\caption{Variance of Score Function estimator}
\label{tab:PhanScore}
\end{table}

Notice that even with substantially larger batch sizes and number of batches
(and hence computational time),
the variance of the score function estimator is orders of magnitude larger
than the frozen phantom estimator.

\subsection{Fast Frozen Phantoms for Tracking Time-Varying MDPs}
\label{sec:fastfrozen}
In Theorem \ref{thm:frozen},  consistency of the frozen phantom
estimator for large
sample size $N$ was established. However, for adaptive control
of constrained MDPs with time varying transition probabilities,
it is necessary to implement the frozen phantoms over 
 small batch sizes of the observed system trajectory
 so that the iterates of the stochastic 
gradient
algorithm are  performed more frequently to track the optimal time varying 
$\a^*$.
The aim of this section is to present an implementation of the frozen
phantoms over short batch sizes $N$ and to show that the resulting
gradient estimate is still consistent.
We call these as ``fast frozen phantoms''.

The implementation of the fast frozen phantom over short batch
sizes proceeds as follows:
Suppose that the gradient $\nablal C(\alpha)$ is to be
estimated using the observed MDP trajectory over the $n$th batch
$I_n\equiv\{ nN +1 ,\ldots, (n+1)N\}$. As in
Theorem~\ref{thm:frozen},
let $\widehat \Gamma_n$ denote
an estimator of
$C(\a)$ using the observed trajectory of the MDP in $I_n$. The 
fast frozen phantom estimators for
the components
$(i,p)$, $i \in S$, $p \in \{1,2,\ldots d(i)\}$,
 of the gradient are (compare with (\ref{GGradEstimator2})):
\begin{multline} 
\label{eq:gradient} 
\widehat{\widehat{G}}_n(i,p) = \frac{2}{N} \biggl[
\sum_{k=nN+1}^{(n+1)N} K_{ip}(\a,k)  
\biggl [ c(i,u_k) - c(i,\tilde u_k)  +
\sum_{j= k+1}^{\min\{(k+\nu(k)), (n+1) N\}} 
\hskip -2em c(Z_j) - 
\nu(k) {\widehat \Gamma}_n 
{\cal D}_n(k) 
\bigr]   \\
+ \sum_{k \in{\cal L}(nN)} K_{ip}(\a,k) 
\biggl(
\sum_{j=nN+1}^{\min\{(k+\nu(k)), (n+1) N\}}
\hskip -2em c(Z_j) - 
\nu(k) {\widehat \Gamma}_n 
{\cal D}_n(k)
\biggr) \biggr]
\end{multline} 
where
${\cal D}_n(k) = \ind{k+\nu(k)\in I_n}$ (phantom $k$ dies  in  $I_n$),
and ${\cal L}(j) = \{ k\le j \colon \nu(k) + k > j\}$ is the list of
living phantoms at stage $j$. A similar estimator holds for the gradient of the constraints.

The interpretation of (\ref{eq:gradient}) is as follows.  At each step
 $k$, the state $Z_k=(i,a)$ is observed and a new phantom system
 (labelled by $k$) is started, generating the phantom decision $\tilde
 u_k$ as described above.  The $k$-th phantom system ``dies'' at the
 hitting time $n= k+\nu(k)$, otherwise it is contained in the set of
 ``living'' phantoms ${\cal L}(n)$.  In a computer program, this
 corresponds to a list.  This phantom will be used to estimate the
 partial derivative of all the functions (cost and constraints) with
 respect to $\alpha_{i,a+1}$ if $a<d(i)-1$. If $a=d(i)-1$ or $a=d(i)$,
 the corresponding phantom system contributes (with opposite signs) to
 the estimation of the gradient w.r.t. $\alpha_{i,d(i)}$.  The difference in costs 
 inside the brackets in (\ref{eq:gradient}) contains the initial
 contribution of a phantom system. Afterwards, while a phantom system
 $k$ is alive, it contributes to (\ref{eq:gradient}) the term
 $c(Z_j)$ at each step, and when it dies ($ {\cal D}_n(k)=1$) it
 contributes the term $\nu(k) \widehat{\Gamma}_n$ (if death occurs within the
 interval $I_n$). The above equation takes only observations of the
 trajectory within the {\em current} interval, thus a final term
 appears considering the contributions of all the living phantoms at
 the start of the interval, because it is possible that phantom
 systems may survive several estimation intervals.
For complete details on the implementation program code and other variations please
see \cite{VKM03}.

\begin{theorem}
\label{cor:bias}
Assume that for any $\a \in \Param^o$,
$\Epi \widehat \Gamma _n = C(\a)$.
 Then the bias of the fast frozen phantom
gradient estimator $\widehat{\widehat{G}}_n(i,p)$ 
of (\ref{eq:gradient}) 
 per batch of size $N$ is given by (with $K_{ip}(\a,k)$ defined in (\ref{eq:kip}))
\begin{equation}
\label{eq:bias}
\Epi [\widehat{\widehat{G}}_n(i,p)] -
\nabla_\a C(\a) = \frac{2}{N}
\sum_{k \colon k+\nu(k)\in I_n}K_{ip}(\a,k)\, \covpi(\nu(k),
\widehat
\Gamma _n).
\end{equation}
\end{theorem}

\begin{proof}
The  proof proceeds in two steps.

Step 1:  
Let $\widehat{\widehat{G}}^o_n(i,p)$ denote the gradient
estimate (\ref{eq:gradient})
when 
$\hat{\Gamma}_N$  is replaced by $C(\a)$.
We first show that the gradient
estimate $\widehat{\widehat{G}}^o_n(i,p)$
is unbiased under the invariant measure $\pi(\param)$, i.e.,
$\Epi[\widehat{\widehat{G}}^o_n(i,p) ] = \nabla_\a C(\a)$.

Notice that under $\pi(\a)$,
consecutive estimators have the same distribution (although they
are not independent), because  the invariant distribution of the number of living
phantoms at the start of the interval is independent of $n$. From the ergodicity of the underlying MDP, $\Epi[\widehat{\widehat{G}}^o_n(i,p)] =
\lim_{n\to\infty} (1/n) \sum_{m=0} \widehat{\widehat{G}}^o_m(i,p)$. Because the estimation by
batches considers breaking up the partial sums of the estimation, then the difference:
\[
{1\over n} \sum_{m=0}^{(n-1)} \widehat{\widehat{G}}^o_m(i,p)
-
\frac{2}{n N} \sum_{k=1}^{nN}
K_{ip}(\a,k) \left[ c(i,a) - c(i,\tilde u_k) + 
\sum_{j= (k+1)}^{k+\nu(k)}  c(Z_j) -
\nu(k) C(\a)
\right]
\]
tends to zero in absolute value, a.s. Using
 Theorem \ref{thm:frozen}, the 
$\frac{2}{n N} \sum_{k}$ term in the above equation
 converges a.s.~to $\nabla_\a C(\a)$ as
$n\to\infty$, for any fixed value of the batch size $N$. This establishes the claim.

Step 2: Using  the above expression consider 
$\widehat{\widehat{G}}^o_n - 
\widehat{\widehat{G}}_n$.
The theorem  follows straightforwardly.
\end{proof}

\subsection{Computational and Memory Complexity}
We present bound on the computational and memory complexity
of the fast frozen phantom algorithm.
Let $\alpha \in\alpha^\mu$. We will bound stochastically the number of living phantoms in terms of Binomial random variables.  

Consider the process $\{Z_n\}$ in stationary operation, and call 
${\cal L}_{nN} (i,a)$ the number of living phantoms in \eq{eq:gradient} that are waiting to hit state $(i,a)$. All these phantoms were created at some 
earlier  time instant
when the chain hit the state $Z_k=(i,p)$ and the phantom decision was $\tilde u_k=a$.  Clearly, the maximum number of phantoms in ${\cal  L}_{nN}(i,a) $ satisfies: 
\[
\|{\cal L}_{nN}(i,a)\| = \sum_{a>p} \sum_{k = t(i,a)}^{nN}\ind{X_k = i; u_k=p, \tilde u_k = a},
\]
where $t(i,a) = \max (j\le nN \colon X_j =1, u_j = a)$, 
because if the state $(i,a)$ is visited at time $j$, at that time all living phantoms that 
were in ${\cal L}_{j-1}(i,a)$ die and ${\cal L}_j(i,a)$ is empty. 
Call $\tau(i,a)$ the return time to state $(i,a)$ and let $n(i)$ 
be the number of visits to state $i$ ($X_k=i$) within two consecutive 
visits to state $(i,a)$. Then the number of phantom systems in 
${\cal L}_{nN}(a)$ is bounded by a Binomial$(\mathbf{p}(a), n(i))$, where 
\[
\mathbf{p}(a) = \frac{\theta_{ia}}{\sum_{m=1}^{p-1} \sin^2(\alpha_{im})},
\]
according to the creation of the phantom systems.  Clearly considering the maximum value of all such probabilities, we can bound the number of phantoms on each list for every value of $\alpha\in\alpha^\mu$.

In \eq{eq:gradient}, the estimator $\widehat{\widehat{G}}^o_n(i,p)$ is composed of  bounded quantities (the state space is finite) plus a contribution of $N$ terms of the order of $ \nu(k)$ each, plus a contribution which is proportional to the random variable:
\[
\sum_{a>p} \sum_{k \in {\cal L}_{nN}(i,a)} \nu(k) \le 
\sum_{a>p}\sum_{i=1}^{n(i)} \nu(k) \le \sum_{a>p}\sum_{k=1}^{\tau(i,a)} \nu(k)
\]
 Note that, as in the proof of  \thm{frozen}, $\nu(k) < \max ( \tau(i,a) : a>p)=\tau$ a.s. 
Because $\alpha^\mu$ is a compact set with ergodic states, the return times are all 
finite a.s.. and $\Ep[\tau(i,a)] =1/ \pi_{ia}$. Therefore, boundedness 
of the $m$-th moment of $\widehat{\widehat{G}}^o_n(i,p)$ now follows from 
boundedness of the $2m$-th moment of $\tau$.

\section{Learning Based Stochastic Gradient
 Algorithms for Constrained MDP}
\label{sec:convergence}
In this section we present  the stochastic gradient algorithms
 that use the parameter free gradient estimators (fast frozen
phantoms) of Sec.\ref{sec:fastfrozen} to optimize
the constrained MDP (\ref{optim}), (\ref{soft}).
Also  weak convergence proofs of these algorithms are presented.
The stochastic algorithms presented are 
stochastic versions of the two deterministic algorithms of
Sec.\ref{sec:lagrange} and  Sec.\ref{sec:multiplier}.
Using the simulation-based  frozen phantom gradient estimator
 \eq{eq:gradient}
 with  local sample averages for the
estimation of the constraint functions may lead to a bias,
thus making the algorithm suboptimal. 
The focus of this section is to
point out the actual bias as well as indications to reduce or
eliminate it.

For notational convenience we  consider equality constraints
here (as  mentioned in Sec.\ref{sec:lagrange} the
inequality constraints can be handled with minor modifications) so that
the constrained MDP problem  (\ref{optim}), (\ref{soft})  reads 
$$
\min_{\param \in \Param^\mu} C(\a),\quad
\mbox{s.t. } \G_l(\a) = 0, \quad l=1,\ldots,L
$$

A control ``agent'' is associated with each of the possible visited
states $i\in S$. The control parameter for this agent is the vector
$(\a_{ip}, p\in\{1,\ldots,d(i)\})$, plus an agent for the (artificial
control) variable representing the Lagrange multiplier $\la$.  The
scheme works by observing the process over a batch size $N$
during which the value of the control parameter does not change.
Over this batch, the constraint is estimated as
$ \hat \G(n)$, see
Sec.\ref{sec:tradeoff},
and  the gradients are estimated
as $\widehat{\nablal C} (n)$, $\widehat {\nablal \G} (n)$
using  \eq{eq:gradient}.
We assume that $\Epi [\hat{\G} (n)] = B(\a)$.
Let
\[
h_0(\a) = \Epi [\widehat {\nablal C }(n)], \quad h_l(\a) = \Epi[\widehat
{\nablal \G_l }(n)], \quad l=1,\ldots, L
\]
be the invariant averages of the batch estimation (refer to 
Theorem \ref{cor:bias}).

\subsection{Stochastic First-Order Primal Dual Algorithm} \label{sec:stochpd}

Consider the stochastic approximation  where the control parameter
$(\a,\la)$ is updated as (c.f.\ (\ref{eq:pd1}), (\ref{eq:pd2}))
\begin{align}
\label{eq:SA1_alpha}
  \a^\ep(n+1) &= \a^\ep(n) -\ep \, \left(\widehat{\nablal C} (n) 
+ \sum_{l=1}^L \left(\la^\ep_l (n) + \rho \hat \G_l(n) \right)\, 
\widehat{\nablal
\G}_l (n) + Z^\mu(n)\right), \quad \a^\ep(0)\in \Param^\mu\\
\la^\ep(n+1) &= \la^\ep(n) + \ep\, \hat \G (n)
\label{eq:SA1_2} 
\end{align}
where the gradient estimators
are given by the fast frozen phantoms in (\ref{eq:gradient}) and
$Z^\mu(n)$ is a projection that ensures that
 $\{\a^\ep(n)\}\in \Param^\mu$. The
above truncation
is a
 mathematical artifice to prove convergence. 
In practical implementation, i.e., when $\ep > 0$,
truncation is not important since one can 
choose $\mu<<\ep  $, e.g., close to the numerical
resolution of the computer, see remark below.

\begin{proposition}
\label{prop:SA1}
For $\a^\ep(n),\la^\ep(n)$ given by (\ref{eq:SA1_alpha}), (\ref{eq:SA1_2}), 
define the interpolated process
$\{\a^\ep(t),\la^\ep(t)\}$ as in (\ref{eq:ainterpolate}), (\ref{eq:linterpolate}).  Then as
$\ep\to0$, $\{a^\ep(t),\la^\ep(t)\}$
converges in distribution to the solution of the ODE (c.f.\ (\ref{eq:odepd}))
\begin{align}
\frac{d}{dt} \a(t) &= -\left[
h_0[\a(t)] + \sum_{l=1}^L (\la_l(t) +\rho\, \G_l[\a(t)]) \, h_l[\a(t)] +
\kappa[\a(t)] 
\right]_\mu , \quad \a(0) \in \Param^\mu\label{eq:pdsa}\\
\frac{d}{dt} \la(t) &= \G[\a(t)], \nonumber
\end{align}
where the notation $[\cdot]_\mu$ refers to the truncated ODE onto $\Param^\mu$
and the added drift is defined as
\begin{equation} \label{eq:kappa}
\kappa(\a) =\rho\, \lim_{n\to\infty} \sum_{l=1}^L\covpi[\hat \G_l (n), \widehat
{\nablal \G_l }(n)].
\end{equation}
\end{proposition}

\begin{remark}
 A weak convergence proof of the stochastic gradient algorithm
(\ref{eq:SA1_alpha}), (\ref{eq:SA1_2})
requires uniform integrability of the gradient estimates.
Without the above truncation, if $\a_{ip}=\pi/2$ or $\a_{ip}=0$,
(equivalently one or more action probabilities $\th_{iu} = 0$), then  the
hitting time $\nu(k)$ in (\ref{eq:nudef}) of the phantom system $k$
with initial state $(i,u)$ is not uniformly bounded and  the
gradient estimator (\ref{eq:gradient}) is  not well defined. 
So in the weak convergence proof below we place
 a $\mu$ size ball around the boundary of $\Param$ (recall
$\Param^\mu$ above
is defined as 
 $\Param$ minus this ball)
and the estimates $\a^\ep(n)$ are truncated to $\Param^\mu$.
However, this truncation is very different to standard truncations
in the stochastic approximation literature -- both
from the algorithm and ODE point of view -- as we now argue.
\\
1.
 Recall from
Sec.\ref{sec:prr}
 that the boundary of the set $\Param$ is a fictitious boundary and 
the action probabilities $\th(\a)$ are
 symmetric about this 
fictitious boundary (since they
are functions of $\sin^2(\a)$ and $\cos^2(\a)$).
For example, suppose that $\a^\ep_{ip}(n) < \pi/2 - \mu$
and that 
(\ref{eq:SA1_alpha})
generates the estimate
$\a^\ep_{ip}(n+1) = \pi/2 + \mu + K$ for some small constant $K>0$.
Then truncation is not required since by symmetry
$\a^\ep_{ip}(n+1) = \pi/2 - (\mu+K) \in \Param^\mu$.
Thus in practical implementation, i.e., when $\ep > 0$,
truncation is not important  since we can 
choose $\mu<<\ep  $. % as the precision of the computer.
The probability that an update lies precisely in a ball
of radius $\mu$ is negligibly small -- since if the estimate overshoots
or undershoots this ball, it is automatically in $\Param^\mu$.
\\
2. Suppose that the untruncated version
of the ODE 
(\ref{eq:pdsa})
has a stable point on the boundary
of $\Param$, e.g., a pure policy. Then clearly, the truncated ODE
will have a stable point within $O(\mu)$ of the stable point of the
untruncated ODE.
Hence the  truncation
is very different to  
standard ODE truncations such as  (\ref{eq:odepd}). 
\end{remark}

\begin{proof}  
We first show that the term in parenthesis on the RHS of 
(\ref{eq:SA1_alpha})
is uniformly integrable. 
Notice that all the terms in the gradient estimate (\ref{eq:gradient})
apart from ${\cal L}(\cdot)$ and $\nu(\cdot)$
are uniformly bounded 
for $\a \in \Param^\mu$ since the MDP is finite state and the batch
size $N$ is fixed and finite. Hence a sufficient condition
for uniform integrability is to show that 
 $\sum_{k \in{\cal L}(nN)} \nu(k)$ in (\ref{eq:gradient}) has finite
variance.

Fix $(i,p)$ for the estimator in (\ref{eq:gradient}). We focus on the
phantoms that have an initial decision $\tilde{u}_k = a$ for a fixed $a$
-- and we call these $a$-phantoms.
By definition,
all $a$-phantoms die simultaneously at time $n_1$
when the process $Z_{n_1} = (i,a)$.
Let $n_0$ denote the previous time instant at which 
$\{Z_n\}$ visited $(i,a)$, i.e., $Z_{n_0} = (i,a)$.
The longest hitting time $\nu(k)$ of all these phantoms is bounded by
the time between successive returns 
$n_1-n_0$.
Consider now the number of living $a$-phantoms in ${\cal L}(n)$ at any time 
$n\in [n_0,n_1]$.
Each of these $a$-phantoms must have been created when the process hits
the state $(i,p)$ and the phantom decision chosen is $a$ -- which happens
with probability $\frac{\th_{ia}}{\sum_{m\ge a} \th_{im}}$.
Therefore an almost sure upper bound for the 
cardinality of ${\cal L}(n)$ is $n_1-n_0$.
Notice that $n_1-n_0$ is the return time of a unichain ergodic MDP on
a finite state and therefore has all moments bounded.

Having established uniform integrability of the updates,
the result follows by direct application of Theorem 5.2.1
in \cite{KC78}.
The continuity of the invariant expectations
follows from the fact that the transition kernel of $Z_n$ is analytic
in $\a$.  To characterize the
drift functions, use:
\[
\Epi   \left[ \hat \G_l (n) \widehat{\nablal \G}_l (n) \right]
=  
\Epi [\hat \G_l (n)]\, \Epi [\widehat{\nablal \G}_l (n) ] + \covpi[\hat \G_l (n),
\widehat {\nablal \G_l }(n)],
\]
which establishes the result. 
\end{proof}

\subsection{Stochastic Augmented Lagrangian Multiplier Algorithm} \label{sec:multsa}

Consider the following stochastic approximation version of the multiplier 
algorithm  \eq{eq:mult2}:
\begin{equation}
\a^\ep(n+1) = \a^\ep(n) - \ep\, 
  \left( \widehat {\nablal C}(n) + \sum_{l=1}^L \bigl(\la_l^\ep(n)+\rho
\hat{\G}(n)\bigr)
\, 
\widehat{ \nablal \G_l} (n) + Z^{\mu}(n)\right), \quad
\a^\ep(0) \in \Param^\mu
\label{UpdateAlpha}
\end{equation}
where 
$Z^\mu(n)$ is a projection that ensures that
 $\{\a^\ep(n)\}\in \Param^\mu$
and 
$\{\la^\ep(n), \ep>0, n\in\Integer\}$ is any tight sequence.
A trivial example is when $\la^\ep(n)$ is a bounded constant (a.s.).

The following result regarding the weak
convergence of (\ref{UpdateAlpha}) is proved in the appendix.

\begin{proposition} 
\label{prop:SA2}Assume  that $\{\la^\ep(n), \ep>0, n\in\Integer\}$ is tight. Define
the interpolated process
$\a^\ep(t)$ of (\ref{UpdateAlpha})
 as in \eq{eq:ainterpolate}.  Then as $\ep\to0$, the interpolated process
$\a^\ep(t)$ converges in distribution to the solution of the ODE:
\begin{equation} \label{eq:alfamult}
\frac{d \a(t)}{dt} = -\left[
h_0[\a(t)] + \sum_{l=1}^L (\bar\la_l +\rho\, \G_l[\a(t)]) \, h_l[\a(t)] +
\kappa[\a(t)]
\right]_\mu, \quad \a(0) \in \Param^\mu,
\end{equation}
where $\bar\la_l$ is an accumulation point of 
the sequence $\{\bar\la(\ep), \ep>0\}$ of (convergent) Cesaro sums:
\[
\bar \la (\ep) \equiv \lim_{N\to\infty} {1\over N} \sum_{n=1}^N \la^\ep(n),
\]
and $\kappa(\a)$ is defined in
(\ref{eq:kappa}).
\end{proposition}

\begin{remark}
If the bias in $h_{l}(\a)$ and $\kappa(\a)$ is negligible, then under no truncation,
 the ODE (\ref{eq:alfamult}) reduces to (\ref{eq:fixedmult}) -- which
is the ODE
for the deterministic fixed multiplier algorithm.
Result \ref{res:mult} of Sec.\ref{sec:multiplier} 
implies that the estimates converge weakly to a near optimal point, provided that the pair
$(\bar\la, \rho)$ is well chosen. The bias in $h_{l}(\a)$ and $\kappa(\a)$ is of order $O(1/N)$.
%which motivates the use of large sample sizes for updates. 
\end{remark}

Consider the following update of the multiplier $\la$ in (\ref{UpdateAlpha}).
Define $\J = \lfloor 1/\ep \rfloor$ and consider the recursion
\begin{equation} \label{eq:slowlam}
\la(n+1) = \la(n) + \bar{\G}(n/\J) \ind{\frac{n}{\J} \in
\Integer}
\end{equation}
together with (\ref{UpdateAlpha}). Thus the multiplier is updated
once every $\J$ time points.
Here 
$\bar{\G}(n/\J) = \frac{1}{\J} \sum_{j=(n-1)\J+1}^{n\J} \hat{\G}(j) $.
If the bias in $h_{l}(\a)$ and $\kappa(\a)$ is negligible, then as $\ep \rightarrow 0$,  the
algorithm (\ref{UpdateAlpha})--(\ref{eq:slowlam}) converges weakly to the
deterministic system (\ref{eq:diffeqn}), (\ref{eq:multnew}) which
is the exact multiplier algorithm. As mentioned in Sec.\ref{sec:multiplier},
this in turn converges to
a local KT point. In a practical implementation, one would
choose $\J$ as a large positive integer. In our numerical examples,
see \cite{GVK03}, even
a choice of $\J=10$ resulted in convergence to a KT point.

\subsection{Tradeoff between Bias and Tracking Ability}
\label{sec:tradeoff}

The three sources of bias in the stochastic gradient algorithm
(\ref{eq:SA1_alpha}), (\ref{eq:SA1_2}) are the bias in the estimates
$\widehat{\nablal \G}$, $\widehat{\nablal C}$ and $\hat{\G}$. 
A quick mathematical artifice  for
eliminating the bias is to use batch sizes $N(\ep)\to\infty$ as $\ep
\rightarrow 0$.  Then the ODE (\ref{eq:pdsa}) becomes identical to
(\ref{eq:odepd}).  In the numerical examples of \cite{GVK03} we chose
$N=1000$ -- for finite $N$ the bias is $O(1/N)$.  Although
choosing $N(\ep)\to\infty$ is theoretically appealing, it is of no
practical use since the stochastic gradient algorithm will not respond quickly to changes in the
optimal policy caused by time variations in the parameters of the MDP.
In  \cite{VKM03} we use batch sizes $N=5,10$
 to update the parameter $\a(n)$ frequently,
 and indeed the stochastic approximation  algorithm can be implemented even
for $N=1$. The bias in the gradient estimator
of Theorem~\ref{cor:bias} can be controlled using
averaging of the estimation of the cost function, or other smoothing
statistical techniques. However, what we really are interested in is
the resulting bias of the stable point of the limiting ODE.
The results of extensive numerical studies
indicate that the bias in $h_{l}(\a)$, $l=0,1,\ldots,L$,
 has  negligible effect  on the behaviour 
of the stochastic gradient algorithm  even for
small batch sizes of $N=5$. For example, we 
performed comparisons using the local
sample average
$
\hat \Gamma_{n} = \sum_{k\in I_{n}} c(Z_{k})/N$ 
and then using the actual theoretical value $\hat \Gamma_{n}=C(\a)$ in $\widehat{\widehat G}_{n}$
of (\ref{eq:gradient}) as well as for
the constraints  with no remarkable difference in the 
stochastic gradient algorithm.

The main source of bias for small batch sizes is that introduced by
$\kappa(\a)$, in Propositions \ref{prop:SA1} and \ref{prop:SA2}. 
Using the local sample average over the $n$-th batch
\begin{equation}
\label{eq:hhat}
 \hat{\G}_l^{\ep}(n) = 
\widehat{\widehat{\G}}_l(n) \defn
\frac{1}{N} \sum_{m\in I_n} \con_l(Z_m) 
\end{equation}
yields a noticeable asymptotic bias $\kappa(\a)$. A better alternative
is to use the Cesaro sum
$\hat{\G}_l^{\ep}(n) = \frac{1}{n} \sum_{k=1}^n 
\widehat{\widehat{\G}}_l(k)$.
Since $\hat{\G}_l^{\ep}(n) \rightarrow \G_l(\a)$ a.s.\ for all $\a \in \Param^o$
as $n \rightarrow \infty$, this Cesaro sum estimator would correct
the asymptotic bias of the stochastic gradient algorithm. However, 
running averages do not respond 
to changes in the underlying parameters (e.g. transition
probabilities) of the MDP since they are decreasing step size
algorithms.  Hence they cannot be used for tracking time varying
optimal policies.  To handle this tracking case, we use in \cite{VKM03} an
exponential smoothing
\begin{align}
  \a^\ep(n+1) &= \a^\ep(n) + \ep \, \biggl(\widehat{\nablal C}(n) - 
 \sum_{l=1}^L \left(\la_l^\ep(n) + \rho \hat{\G}_l^{\ep}(n) 
\right)\, \widehat{\nablal\G}_l(n) \biggr)\label{eq:spd1b}\\
\la_l^\ep(n+1) &=  \la^\ep(n) + \ep\,\hat{\G}_l^{\ep}(n) , \quad
\hat{\G}_l^{\ep}(n+1) =
\hat{\G}_l^{\ep}(n) + \delta 
\left(\widehat{\widehat{\G}}_l(n) - \hat{\G}_l^{\ep}(n)\right), \quad l=1,\ldots, L,
\label{eq:spd3}
\end{align}
where $\delta>0$ and $\widehat{\widehat{\G}}_l(n)$ is the local
sample average in (\ref{eq:hhat}).
Using a
two time scale stochastic approximation argument it can be shown
that if $\epsilon/\delta \rightarrow 0$, e.g., if $\delta =
\sqrt{\epsilon}$, and $\epsilon \rightarrow 0$, then the asymptotic limit points
of the corresponding ODE are
unbiased.  In practical implementation, for non zero $\delta$, the
estimates are biased.
While the asymptotic bias can be controlled, the exponential smoothing
delays the reaction time of the stochastic gradient algorithm since a faster
time scale has been introduced, as illustrated in the following numerical
example.

\noindent {\bf Numerical Example}:
We consider adaptive stochastic control of the following
time-varying constrained MDP: 
For time up to $4000$,
$S = \{0,1\}$, $\actionset_i = \{0,1,2\}$, $i \in S$ (i.e., $d(0) = d(1) = 2$),
\[
A(0) = 
\begin{pmatrix}
0.9 & 0.1 \\
0.2 & 0.8 
\end{pmatrix}, \quad
A(1) =
\begin{pmatrix}
0.3 &  0.7\cr
0.6 & 0.4 \cr
\end{pmatrix}, \quad A(2) = 
\begin{pmatrix}
0.5 &  0.5\cr
0.1 & 0.9 \cr
\end{pmatrix}.
\]
The cost matrix $(c(i,a))$, two  constraints ($L=2$)   matrices
$(\con_1(i,a))$, $(\con_2(i,a))$ 
are
$$ (c(i,a)) = -\begin{bmatrix}
50 &    200 &   10 \\
3 &   500 &    0 \end{bmatrix}, \quad
\con_1= \begin{bmatrix}
20 & 100 & -8 \\
-3 & 4 & -10 \end{bmatrix},\;
\con_2= \begin{bmatrix}
10 & -20 & 22 \\
-19 & 17 & -15 \end{bmatrix}.$$
The optimal control policy  incurs a cost of -111.80 (or equivalently
a reward of 111.80) and is  randomized with
probabilities (\ref{eq:theta*})
$$ \th^* = \begin{bmatrix} 0  & 0.2 & 0.8\\
  0 & 0.28 & 0.72 \end{bmatrix}.$$
 For time between 
4000 and  12000 the transition probabilities
are
  $$ A(0) = \begin{bmatrix}
 0.5 &  0.5\\
 0.5 & 0.5 
 \end{bmatrix}, 
  \qquad A(1)=
 \begin{bmatrix} 0.9 &  0.1 \\
 0.1 & 0.9 \end{bmatrix},
  \qquad A(2) =
 \begin{bmatrix}0.5  & 0.5 \\
 0.45 & 0.55 \end{bmatrix}.$$
This has an optimal cost of -44.52 (i.e. reward of 44.52).

The algorithm was initialized
with randomized policy
$$\th(0) =\begin{bmatrix}
0.1 &  0.1 &  0.8\\ 0.0&   0.2&  0.8 \end{bmatrix}.$$
The batch sizes over which the gradients
are estimated was
chosen as $N=10$. The parameters used  in the primal dual algorithm are
 $\rho = 100$, $\ep = 2\times 10^{-7} $ (see (\ref{eq:SA1_alpha}),
(\ref{eq:SA1_2})). 

As can be seen from  Fig.\ref{fig:pdtrack}, it takes
only around  100 batches 
(1000 time points) for  the algorithm to rapidly approaches the optimal
policy. The algorithm also quickly responds to the change in optimal
policy at batch time 400.
The choice of the discounting factor $\delta$ in (\ref{eq:spd3})
of the primal dual
method clearly shows the trade off between bias and tracking ability
in Fig.\ref{fig:pdtrack}.
For $\delta=1.0$, the algorithm has fast tracking properties but a
large bias. For $\delta=0.5$ and $\delta=0.1$ the
bias gets smaller but the adaptation rate is slower.

Our conference  paper \cite{VKM03} and report \cite{GVK03} give several
other numerical examples with small batch sizes for the projected
gradient and multiplier
algorithm.

 \begin{figure}[h]
 \centering
 \setlength{\epsfxsize}{4in}
 \epsfbox{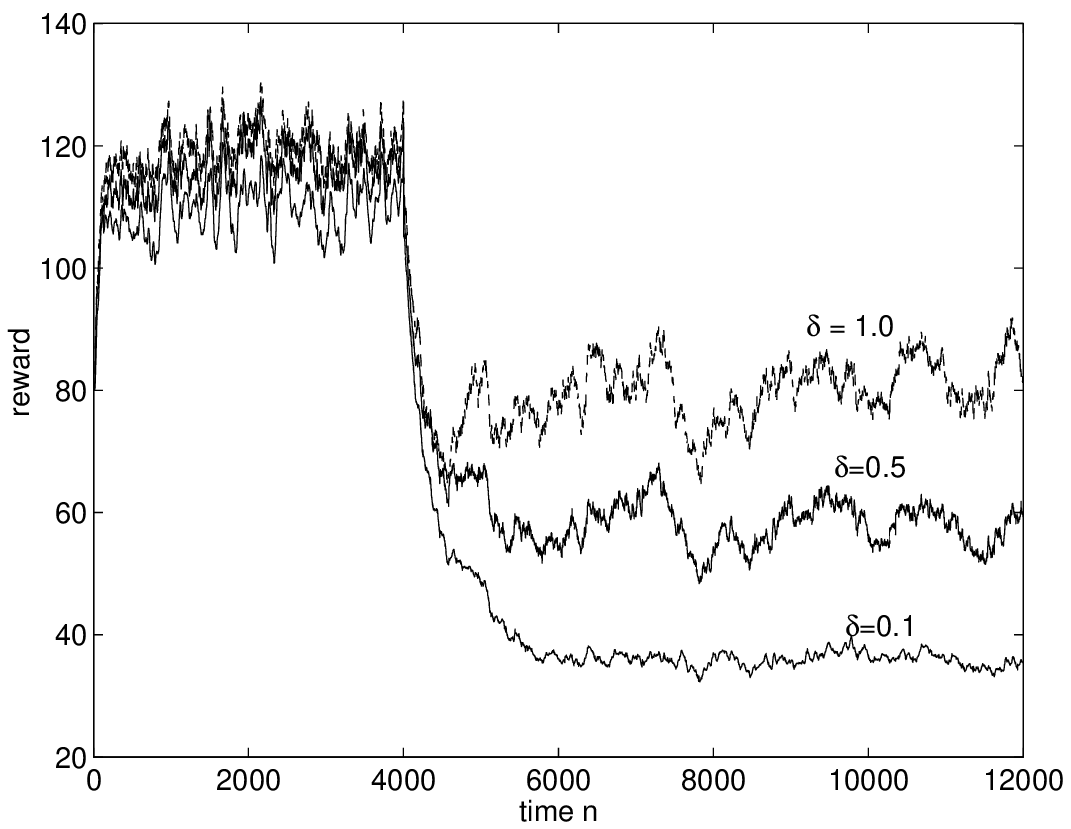}
 \caption{Primal dual algorithm based  stochastic adaptive controller}
 \label{fig:pdtrack} 
 \end{figure}

 \section{Case Study: Monotone Policies for Packet Transmission Scheduling over Correlated Wireless Fading Channels}
\label{sec:tx_sched}
 In this section we consider a 
special case of a constrained MDP
that arises in transmission scheduling in wireless telecommunication
systems. The constrained MDP we consider, models a transmission scheduling problem
in a wireless telecommunication network.
The action set is $\mathcal{U} =\{0,1\}$ corresponding to {\em transmit} and 
{\em do not
transmit}, respectively.
By using a Lagrangian formulation for dynamic programming,
we show that the optimal policy is a randomized mixture between two
deterministic monotone
(threshold) policies; such a policy is a two-step staircase function as plotted in Fig.~\ref{fig:staircase}.  So  the weak derivative
based stochastic approximation algorithms presented above
can  be used to estimate this structured optimal policy. Because of the threshold
structure of the optimal policy, the algorithm implementation is very
efficient. 

Consider the following transmission scheduling problem over a correlated fading wireless channel. At each time slot, a user has to decide whether to transmit a packet unless the packet storage buffer is empty. The objective is to minimize the infinite horizon average transmission cost subject to a constraint on the average delay penalty cost. As in \cite{WM95,ZK99}, we model the correlated fading wireless channel by a finite state Markov chain (FSMC). That is, we assume the channel state evolves according to a FSMC, and the channel state realization is known at every time slot. At time $n = 0,1,\ldots$, the system state is the 3-tuple
$X_{n}=[ X^b_n,
X^c_n, X^y_n]$, where:\\
(i)  $X^b_{n}\in \mathcal{B} = \{0,1,\ldots\}$ is the buffer occupancy state
\\ (ii) $X^c_n$ 
 is the  state of the correlated wireless communication channel.
Assume $X^{c}_n\in \q=\{\q_{1},\q_{2},\ldots,\q_{K}\}$, where $\q$ is the finite channel state space and $\q_{i}$ corresponds to a better channel state than $\q_{j}$ for all $i>j$; $X^c_n$ evolves as a Markov chain with the transition probabilities $A^c(X^c_{n+1}=\q_{j}|X^{c}_n= \q_{i})=a^c_{ij}$.
\\
(iii)  $X^y_n$ is the number of new packets arriving at the buffer.
For simplicity, assume an i.i.d binary packet arrival process, that is at any time $n$ $X^y_n \in \mathcal{Y} =\{0,1\}$, which denotes either
the arrival of no packet or one packet, with the probability mass function
  $P(X^y_n = 1)=\delta$ and $P(X^y_n = 0) =1-\delta$.

  The system state space is then the countable set $\mathcal{S} =  \mathcal{B} 
  \times \q\times \mathcal{Y}$

Let the action sets be $\mathcal{U} = \{0,1\}$ for every buffer occupancy state $X^b>0$ and $\mathcal{U}_0 = \{0\}$ for the buffer occupancy state $X^b=0$, where $0$ and $1$ correspond to the action of not transmitting and transmitting respectively.

We now define the costs and constraints of constrained MDP in Problem:
\begin{itemize}
\item  The transmission cost function  $c(\cdot,\cdot):\q \times \mathcal{U} \rightarrow \mathbb{R}_{+}$ is a function  of the channel state.  Assume that when a transmission is not attempted, no transmission cost  is incurred, that is $c(\cdot,0)=0$. 
 \item The constraint is specified by a delay penalty cost  $\beta(\cdot,\cdot): \q\times \mathcal{U} \rightarrow \mathbb{R}_{+}$, which is applicable only for buffer occupancy state $x^b>0$.  Assume that when a transmission is attempted, there is no delay penalty cost, that is $\beta(\cdot,1) = 0$.  
 
 For channel utilization enhancement, it is assumed that $c(\cdot,u)$ is decreasing and $\beta(\cdot,u)$ is increasing in the channel state, that is the transmission cost is lower and the delay penalty cost is higher for better channel states.   \end{itemize}

If the transmission of a packet over the channel is attempted, that is action $u=1$ is selected, a packet will be successfully received (and hence removed from the buffer) with probability given by the function 
$f:\q \rightarrow [0,1]$.
Here,  $f(\cdot)$ is a user-defined increasing function, that is, a higher (better) channel state has a higher success probability. 

The transition probabilities of the constrained MDP are then given by  
\begin{align*}
 \mathbb{P}(X_{n+1}|X_n,u=0) & = \mathbb{P}(X^c_{n+1}|X^c_n) \mathbb{P}(X_{n+1}^y)  \mathbf{I}(X^b_{n+1} = X^b_n + X^y_n) \\
 \mathbb{P}(X_{n+1}|X_n,u=1) & = \mathbb{P}(X^c_{n+1}|X^c_n) \mathbb{P}(X_{n+1}^y)  f(X^c_n)  \mathbf{I}(X^b_{n+1} = X^b_n + X^y_n - 1) \\ & + \mathbb{P}(X^c_{n+1}|X^c_n) \mathbb{P}(X_{n+1}^y) (1-f(X^c_n))  \mathbf{I}(X^b_{n+1} = X^b_n + X^y_n),
\end{align*}
where $\mathbf{I}(\cdot)$ is the indicator function.
%\textbf{CMDP formulation:} 
The corresponding constrained MDP is then given by (\ref{J}),~(\ref{costconstraint}) with only one constraint, that is, $L=1$: \begin{align} \label{eq:constraint} \lim \limits_{N \rightarrow \infty} \sup \frac{1}{N} \mathbb{E}_u \biggl[ \sum \limits_{n=1}^N \beta(X^c_n,u_n) \biggr] \leq \gamma. \end{align}

In the remainder of the section we will outline the steps involved in proving the threshold structure of the optimal policy of the above constrained MDP and describe how the threshold structure can be exploited in the proposed weak derivative based stochastic approximation algorithm. 
\begin{figure}
\centering
\includegraphics[scale=0.6]{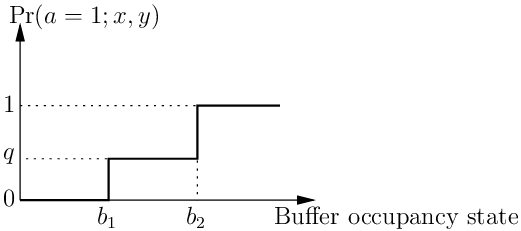}
\caption{The constrained average cost optimal policy  $\mathbf{u}^{*}([x^b,x^c,x^y])$ is a randomized mixture between two policies that are deterministic and monotonically increasing in the buffer occupancy state $x^b$.}
\label{fig:staircase}
\end{figure}

The steps involved in proving the structural result for the considered countable state, infinite horizon average cost constrained MDP includes
\begin{itemize}
\item Derive a condition for which all constrained policies induce a stable buffer and recurrent Markov chains.
\item Use the Lagrange multiplier formulation and prove the existence of an unconstrained optimal policy under the condition for buffer stability. 
\item Prove the threshold structure of the unconstrained (Lagrangian cost) optimal policies by using the supermodularity concept. The structure of the constrained optimal policy then follows due to a well-known result that relates unconstrained and constrained optimal policies \cite{BR85,Alt99}. 
\end{itemize}

The condition for buffer stability and recurrence of the Markov chains is as follows,
see \cite{NK09,NK10} for proof.
\begin{lemma}
\label{lm:stability}
Denote $\min \limits_{x^c \in \q}\{ f(x^c)\} = \underbar{f}$\, ; $\min \limits_{x^c\in \q} \beta(x^c,0) = \underline{\beta}$. If 
$
  \frac{\delta }{\underbar{f}} < 1- \frac{\gamma}{\underline{\beta}}  
$,
then every policy $\mathbf{u}$ satisfying the constraint (\ref{eq:constraint})  induces a stable buffer, and a recurrent Markov chain.%  As a result there exists a stationary optimal policy for (\ref{eq:cmdp}).
\end{lemma}

\subsubsection*{Lagrange formulation and existence of an optimal policy}
We convert the constrained MDP to an unconstrained MDP by the Lagrange multiplier method. In particular, for a Lagrange multiplier $\lambda$, the instantaneous Lagrangian cost at time $n$ is 
$ c(X_n,u;\lambda) = c(X_n^c,u) + \lambda \beta(X_n^c,u)$.
The Lagrangian average cost for a policy $\mathbf{u}$ is then given by
\begin{align}
J_{x_0}(\mathbf{u};\lambda) = \lim \sup \limits_{N \rightarrow \infty} \frac{1}{N} \mathbb{E}_{\mathbf{u}} \biggl[ \sum \limits_{n=1}^{N}c(X_n,u_n;\lambda)|X_0 = x_0 \biggr],
\label{eq:Lcost}
\end{align}
and the corresponding unconstrained MDP is to minimize the above Lagrangian average cost.

The existence and threshold structure of an unconstrained stationary average Lagrangian cost optimal policy are established by viewing the average cost MDP model as a limit of discounted cost MDPs with discount factors approaching $1$. In particular, \cite{Ros83,Sen89} provide the theory for relating average cost optimal policies to discounted cost optimal policies. Define the discounted cost as below
\begin{align}
J_{x_0}^{\nu}(\mathbf{u};\lambda) = \lim \limits_{N \rightarrow \infty}\mathbb{E}_{\mathbf{u}} \biggl[ \sum \limits_{n=1}^{N} \nu ^{n-1} c(X_n,u_n;\lambda)|X_0 = x_0 \biggr],
\label{eq:Lcost_discount}
\end{align} 
where $0\leq \nu \leq 1$ is the discount factor. Define the optimal discounted cost by $ V^{\nu}(x_0) = \inf \limits_{u \in \mathcal{D}} J_{x_0}^{\nu}(\mathbf{u};\lambda)$ (for notational convenience we omit the notation of Lagrange multiplier $\lambda$ in $V^{\nu}(x_0)$).

In \cite{Ros83,Sen89}, the authors proved that the  average cost optimal policy exists and inherits the structure of the discounted cost optimal policies under the following conditions:\\
A1. For each state $x$ and discount factor $\nu$, the optimal discounted cost $V^{\nu}(x)$ is finite. 
\\
 A2. Assume a reference state $0$. There exists a nonnegative $N$ such that $-N \leq h_{\nu}(x) \stackrel{\triangle}{=} V^{\nu}(x) - V^{\nu}(0)$ for all $x \in \mathcal {S}$ and $\nu \in (0,1)$. 
 \\
 A3. There exists nonnegative $M_{x}$, such that $h_{\nu}(x) \leq M_{x}$ for every $x \in \mathcal {S}$ and $\nu$. For every $x$ there exists an action $u(x)$ such that $\sum\limits_{x'}
\mathbb{P}(x'|x,u(x))M_{x'} < \infty$. 

Define the reference state by $X^0=[x^b=0,x^c=\q_{K},x^y=0]$. In light of Lemma~\ref{lm:stability} it is clear that the policy of always transmitting whenever the buffer is not empty will induce a stable buffer, and hence finite expected time and cost for first passage to the reference state. As a result, due to Propositions 5(i) and 4(ii) in \cite{Sen89}, A1 and A3 hold. Furthermore, as all instantaneous costs are bounded, the following value iteration converges to the optimal discounted cost $V^{\nu}(\cdot)$ for all discount factor $\nu$
\begin{align}
&V_{n+1}^{\nu}(x) = \min \limits_{u \in \mathcal{U}} Q^{\nu}_{n+1}(x,u) \label{eq:value_iter1} 
\\
&Q^{\nu}_{n+1}(x,u) = c(x,u;\lambda) + \nu \sum \limits_{x' \in \mathcal{S}} \mathbb{P}(x'|x,u) V_n^{\nu}(x',a). 
\label{eq:value_iter2}
\end{align}
Using the recursion (\ref{eq:value_iter1})--(\ref{eq:value_iter2}) it can be shown by induction that $V^{\nu}([x^b,x^c,x^y])$ is increasing in $x^b$ and $x^y$ \cite{NK09,NK10}, which implies that A2 holds. 
%\begin{lemma}
%\label{lm:vmono}
%\textit{For any discount factor $\alpha \in (0,1)$, the optimal discounted cost $V_{\alpha}([b,x,y])$ is increasing in the buffer state component $b$, the number of new packets $y$ and decreasing in the channel state$x$.}
%\end{lemma}
%\begin{proof}
%Since all instantaneous costs are bounded, the value iteration algorithm converges for any discount factor $\alpha \in (0,1)$. Hence, the sequence $V_{\alpha}^{t}([b,x,y]):t=0,1,\ldots$ generated by 
%\begin{align}
%V_{\alpha}^{t+1}([b,x,y]) = \min_{a\in \mathcal{A}_{[b,x,y]}} & \Bigl[c([b,x,y],a;\lambda) + \alpha \sum \limits_{ x' \in \mathcal{X},y' \in \mathcal{Y}}P_{\Gamma}(x'|x)P_{Y}(y') \nonumber \\  & \Bigl( f(x,a)V^{t}(b+y-a) + (1-f(x,a))V^{t}(b+y) \Bigr) \Bigr]
%\end{align}
%converges for any initialization of $V_{\alpha}^{0}([b,x,y])$. Select $V_{\alpha}^{0}([b,x,y])$ that is increasing in $b,y$, decreasing in $x$. Then it is clear by induction that $V_{\alpha}^{t}([b,x,y])$ is increasing in $b,y$ and decreasing in $x$ for all $t\geq 1$. Hence  $V_{\alpha}([b,x,y])$ is increasing in $b,y$ and decreasing in $x$.
%\end{proof}

\subsubsection*{Threshold structure of discounted/average cost optimal policies}
Due to convergence of (\ref{eq:value_iter1}),~(\ref{eq:value_iter2}) for all initial conditions, in order to show that the discounted cost optimal policy is monotonically increasing in the buffer state $b$ it suffices to show $Q^{\infty}([x^b,x^c,x^y],u)$ is submodular in $(x^b,u)$ for all $x^b\geq 1$ for some initial condition \cite{Top98}. This can be done via mathematical induction as in the theorem below, see \cite{NK09}
for proof. We also refer to \cite{HK10} where the Nash equilibrium of a switching control access control problem has a monotone structure.

 \begin{theorem}
\label{thm:threshold}
\textit{ The discounted (Lagrangian) cost optimal policy is a threshold policy of the form}
\begin{align}
\label{eq:threshold}
u^{*}_{\nu}([x^b,x^c,x^y]) = \left \{ \begin{array}{clcr} 
0 & \textup{if } 0 \leq x^b < b(x^c,x^y) \\
1 & \textup{if } b(x^c,x^y) \leq x^b,
\end{array}
\right.
\end{align}
\textit{where $b(x^c,x^y):\q \times \mathcal{Y} \rightarrow \mathcal{B}$ defines the threshold for the pair of channel state and packet arrival event $(x^c,x^y)$.}
\end{theorem}

 Therefore (unconstrained) Lagrangian average cost optimal policy, which inherits the threshold structure of some sequence of discounted cost optimal policies, is of the form (\ref{eq:threshold}). Due to Theorem 4.3 in \cite{BR85}, the constrained optimal policy for the  constrained MDP is a randomized mixture of two threshold policies:
 %as sketched in Fig.~\ref{fig:staircase}
\begin{align}
\mathbf{u}^{*} = q(x^c,x^y) \mathbf{u}^{*}_1 + (1-q(x^c,x^y)) \mathbf{u}^{*}_2
 = \begin{cases} 0  & \text{ if }   0 \leq x^b < b_1(x^c,x^y)  \\
 q(x^c,x^y)   & \text{ if } b_1(x^c,x^y) < x^b < b_2(x^c,x^y) \\
 1  & \text{ if } x^b > b_2(x^c,x^y). \end{cases}
%\quad \text {where } 0 \leq q \leq 1.
\label{eq:mixture}
\end{align}
Here for each channel state $x^c\in \q$ and packet arrival state $x^y \in \mathcal{Y}$,   
$q(x^c,x^y) \in [0,1]$ denotes the mixture probability and $\mathbf{u}^{*}_1,\mathbf{u}^{*}_2$ are monotone policies in the buffer state
$x^b$ of the form (\ref{eq:threshold}) with
threshold states $b_1(x^c,x^y)$ and $b_2(x^c,x^y)$, respectively.
 Therefore, the  optimal policy $\mathbf{u}$  has a  simple 
 threshold structure.

 \section{Perspective. Bias and Variance of Gradient Estimators} \label{sec:biasvarwd}
In this final section, we briefly compare the statistical properties of the score function and weak derivative gradient estimators 
discussed above; see also \cite{Kri16} for a detailed discussion and proofs.

\begin{theorem} \label{thm:scoremarkov}
For a Markov chain with initial distribution $\belief_0$,  regular transition matrix $\tpm$ with coefficient of ergodicity $\dob$  and stationary distribution $\belief_\theta$:
\begin{compactenum}
\item The score function gradient estimator  has: % the following properties:
bias $  O(1/\bsize)$ and 
variance $ O(\bsize) $ 
\item The weak derivative gradient estimator has
bias $ O(\dob^m) \, \dvar{\belief_0}{\belief_\theta} + O(\dob^\bsize) $ and 
variance $O(1) $ % (\beliefm^\p \cost^\p )^2   (\beliefm^\p S^\model_n)$
\end{compactenum}
\end{theorem}
The result shows that  despite the apparent simplicity of the score function gradient estimator (and its widespread use), the weak derivative estimator performs better in both bias and variance.
The variance of the score function estimator  actually
grows with sample size! This is apparent from the numerical examples presented in Sec.\ref{sec:numerical}
where the derivative estimator has a substantially smaller variance
than the score function gradient estimator.

Why is  the variance of the score function gradient estimator $O(N)$ while the variance of the weak derivative estimator is $O(1)$?
The weak derivative estimator uses the difference of two sample paths. Its variance
is dominated by a term of the form $\sum_{m=1}^N g^\p (\tpm^m)^\p (\belief_0 - \bar{\belief}_0)$ where $\belief_0$ and $\bar{\belief}_0$ are the
initial distributions of the two trajectories. This sum is bounded by $\text{constant} \times  \sum_{m=1}^N \rho^m $ which is $O(1)$ since  $\rho < 1$ for a regular
transition matrix $\tpm$.
In comparison, the score function estimator uses a single sample path and its variance is dominated by a term of the form  $\sum_{m=1}^N g^\p (\tpm^m)^\p \belief_0 $.
This sum grows  as $O(N)$. The proof in \cite{Pfl96,Kri16} formalizes  this argument.

To gain additional insight  consider the following simplistic examples.

1. I.I.D.\  process.  Suppose $\tpm = 1 \belief_\theta^\p$.   If we naively use the score function estimator for a Markov chain, the variance is $O(N)$.
In comparison, the weak derivative estimator is identical to that of a random variable since $\tpm$ has identical rows. So the variance of the weak derivative estimator is substantially smaller.

2. Constant cost $c(x) = 1$.  In this trivial case, $\nabla_\theta \esp_{\belief_\theta}\{c(x)\} = 0$.   If we naively use the score function estimator for a Markov chain,
the variance is $O(N)$. In comparison the weak derivative estimator yields 0 implying that the variance is zero. So for nearly constant costs,
the weak derivative estimator is substantially better than the score function derivative estimator.

\section{Conclusions and Extensions}

In this paper simulation based gradient algorithms have been presented
for adaptively  optimizing a constrained average cost finite state
MDP. First a 
parameterization of the randomized control policy using
spherical coordinates was presented.
Then  a novel measure-valued gradient estimator using
frozen phantoms was presented. 
The frozen phantoms were  based on
``cut and paste'' techniques which look at the past observation
history.  By filtering these frozen
phantoms, we derived  a parameter free consistent algorithm for
estimating gradients of the cost and constraints
without explicit knowledge of the
transition probabilities
of the MDP and without requiring off-line simulations.
A fast version of the frozen phantom estimator suitable for
adaptive control of 
MDPs with slowly time varying parameters was also given.
As illustrated in
Sec.\ref{sec:numerical},
the resulting  gradient
estimator has much smaller variance than score function
based gradient estimators. The measure-valued  gradient estimator was then
used in a stochastic gradient algorithm with  fixed step
size in order to track time varying MDP with
unknown transition probability matrices.
Primal dual and multiplier based stochastic gradient algorithms were
presented for handling the constraints.
These algorithms are nearly optimal in that they converge
weakly to a local minimum with a bias --  this bias
is identifiable and can be made negligible.

In  \cite{VKM03}, a detailed 
numerical study of the frozen phantoms (parameter free gradient
estimators) is conducted. The effect of moving averages and exponentially
discounted averages on the bias is also studied. See also \cite{GVK03,VKB02}
for further numerical examples.
In current work, we are examining applications of the
techniques in this paper to admission control of wireless networks.
As mentioned in Sec.\ref{sec:intro}, in this case the quality of service and
blocking probability 
constraints naturally translate into constraints on the MDP.

Given that the proposed adaptive controller is a fixed step size
stochastic approximation algorithm, several variations such as iterate
averaging \cite[Chapter 11]{KY03}, adaptive step size updating
\cite[Sec.3.2]{KY97} and
decentralized asynchronous implementation \cite[Chapter 12]{KY03}
are possible. It is also
worthwhile examining the use of similar methods for partially observed
MDPs (POMDPs). For example \cite{KD07,Kri11,Kri13,KAB18} use offline policy gradient algorithms that use the SPSA algorithm; there is strong motivation to develop
online policy gradient algorithms that use weak derivatives. Also \cite{HH04} proposes an interesting example of a combined score function weak derivative method.

In the numerical examples, we considered the case where the online policy gradient algorithm  tracks the optimal policy which jump changes. As show in \cite{YKI04,KTY09,NKY17} if the optimal policy itself jump changes according to a Markov chain with transition probability  matrix  $I + \epsilon Q$ where $Q$ is a transition rate matrix, where $\epsilon$ is the same order of magnitude as the step size of the policy gradient algorithm,   then the weak convergence analysis has an interesting form: the averaged system converges to a Markov modulated ordinary differential  equation (ODE) rather than a deterministic ODE.

\bibliographystyle{abbrv}

\bibliography{C:/Users/vikramk/styles/bib/vkm}

\section{Appendix: Proof of Proposition \protect\ref{prop:SA2}}

\begin{proof} 
First, from tightness of $\{\la^\ep(n)\}$, it follows that the family $\{\bar\la(\ep)\}$ of
(deterministic) averages lies in a compact set, thus the accumulation points $\bar\la$
exist.

 From the proof of uniform integrability in Proposition \ref{prop:SA1},
and 
 tightness of $\{\la^\ep(n)\}$, it follows that
the sequence
$\{(\a^\ep(n+1)-\a^\ep(n)/\ep)\}$ is uniformly integrable, which implies that
$\{\a^\ep(n)\}$ is tight. Therefore for any sequence $(\a^{\ep_k}(\cdot),
\la(\ep_k))$ there is at least one (weakly) convergent subsequence with a.s.~Lipschitz
continuous limit (refer to \cite{KY03}). For the rest of the proof, until specified,
assume that $\ep$ labels a weakly convergent subsequence (to avoid the $\ep_k$
cumbersome indexing). We will now identify the limits of such convergent subsequences
and show that they all satisfy the same ODE. Also to ease the notation in the proof, call
$Y _0(n) = \widehat {\nablal C }(n), Y _l(n) = \widehat
{\nablal \G_l }(n)$.

From the definition \eq{eq:ainterpolate} it follows that:
\[
\a^\ep(t+s) -\a^\ep(t) = - \ep \sum_{n=\lfloor t/\ep\rfloor}^{\lfloor
(t+s)/\ep\rfloor-1} 
 \left[ Y _0(n) + \sum_{l=1}^L \la_l^\ep(n+1) \, Y _l(n) \right].
\]

Divide now the interval $(t, t+s]$ into subintervals of small size $\delta_\ep$ containing
each a number $n_\ep$ of updates, as shown in \fig{averaging}. 

\begin{figure}[h] 
\centering
\setlength{\epsfxsize}{4.5in}
\epsfbox{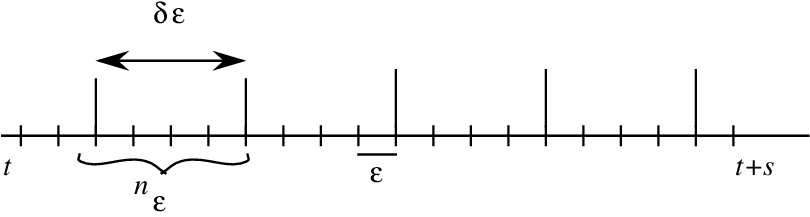}
\caption{$\delta_\ep=\ep n_\ep$. Condition: $\delta_\ep\to0, n_\ep\to\infty$ as
$\ep\to0$.}
\label{fig:averaging}
\end{figure}

Using the $\delta_\ep$ grouping of subintervals, one obtains the telescopic sum:
\[
\a^\ep(t+s) -\a^\ep(t) = -  \sum_{l=\lfloor
t/\delta_\ep\rfloor}^{\lfloor (t+s)/\delta_\ep\rfloor-1}  \delta_\ep \times\left( {1\over
n_\ep} \sum_{j=l n_\ep}^{(l+1) n_\ep -1}
 \left[ Y _0(j) + \sum_{l=1}^L \la_l^\ep(j+1) \, Y^\ep_l(j) \right]
\right).
\]

We now show that if $\Field^\ep(t)$ denotes the $\sigma$-algebra generated by the
interpolated process up to time $t$, then:
\begin{eqnarray}
&&\E [\a^\ep(t+s) - \a^\ep(t) \given \Field^\ep(t)] \approx \nonumber\\
&&\hskip 1em \sum_{l=\lfloor t/\delta_\ep\rfloor}^{\lfloor (t+s)/\delta_\ep\rfloor-1}
\delta_\ep \left(
h_0[\a^\ep(l n_\ep)] + \sum_{l=1}^L (\bar\la_l(\ep) + \G_l[\a^\ep(l
n_\ep)])\, h_l[\a^\ep(l n_\ep) + \kappa[\a^\ep(l n_\ep)]
\right) 
\label{eq:averaging}
\end{eqnarray}
where the expectation of the absolute error in the
end points of the $\delta_{\ep}$ discretization 
 vanishes as
$\ep\to0$. Let $\esp_{l n_\ep}$ denote the expectation conditioning on the information
available up to the start of the current small subinterval of size $\delta_\ep$.  Use now
conditional expectations to express
$\esp[Y _l(j) \given \Field^\ep(t)] =\esp[ \esp_{l n_\ep}[Y _l(j) ]
\given\Field^\ep(t)], l n_\ep \le j < (l+1) n_\ep $ for each term in the telescopic sum
$l=0,\ldots, L$. That is, we
use a filter of the terms, focusing on each of the averages within subintervals.

Because one is interested in averages, any version of the process $\{\a^\ep(n)\}$
can be used to characterize these conditional expectations. In particular, Skorohod
representation establishes that there is a process $\tilde\a^\ep(n)$ for each $\ep$
in the weakly convergent subsequence, such that $\tilde\a^\ep(n)$ has the same
distribution as $\a^\ep(n)$ and it converges with probability 1 to the same
a.s.~continuous limit $\a(t)$ (see \cite{KY03}). Because $\a(t)$ is
Lipschitz continuous w.p.1, $\|\a(l \delta_\ep+\delta_\ep) - \a(l\delta_\ep)\| =
{\cal O}(\delta_\ep)$ and since $\tilde\a^\ep(n)$ converges w.p.~1, it follows that 
$
\sup_{l n_\ep\le j < (l+1) n_\ep} \| \tilde \a^\ep(j) - \a(l \delta_\ep) \| \to 0
\quad\mbox{a.s.}$
which implies that the underlying distribution of  the batch estimators $Y_l (j), j=l
n_\ep, \ldots, (l+1) n_\ep -1$ converges to that of the fixed-$\param$ MDP at the
parameter value $\param=\a(l\delta_\ep)$. Using the fact that $\a^\ep(l
n_\ep)$ converges in distribution to $\a(l\delta_\ep)$, it follows that:
\begin{eqnarray*}
&&\esp_{l n_\ep} \left[
 {1\over n_\ep } \sum_{j=l n_\ep}^{(l+1) n_\ep -1}
 \left[ Y _0(j) + \sum_{l=1}^L \la_l^\ep(j+1) \, Y _l(j) \right]
 \right]
\\
&& \hskip 2 em =
\esp_{l n_\ep} \left[
 {1\over n_\ep } \sum_{j=l n_\ep}^{(l+1) n_\ep -1}
 \Ep\left[ \widehat{\nablal C} _0(j) + \sum_{l=1}^L (\la_l^\ep(j) + \hat
\G (j))\, \widehat{\nablal \G _l}(j) 
\right]
 \right ].
\end{eqnarray*} 

Our  batch estimation procedure  takes into account only new
information on each estimation interval. Given the initial state value (with the
aggregated  information about the living phantoms), and the value of $\la^\ep_l(j)$,
the expectation for the fixed $\param$ process of the gradient estimator
$\widehat{\nablal \G _l}(j)$ is independent of $\la^\ep_l(j)$. As
$n_\ep\to\infty$, the underlying process
$\{Z_n\}$ will have the stationary distribution for the $j$-th estimation batch, so that:
\begin{eqnarray*}
&&\esp_{l n_\ep} \left[
 {1\over n_\ep } \sum_{j=l n_\ep}^{(l+1) n_\ep -1}
 \left[ Y _0(j) + \sum_{l=1}^L \la_l^\ep(j+1) \, Y _l(j) \right]
 \right ]\\
&& \hskip 2 em = 
h_0[\a^\ep(l n_\ep)] + 
\sum_{l=1}^L  \left(
h_l[\a^\ep(l n_\ep)]\,
\esp_{l n_\ep} \left[{1\over n_\ep } \sum_{j=l n_\ep}^{(l+1) n_\ep -1}
  \la_l^\ep(j) \right]
+ \rho \, \Epi[\hat \G_l (n) \, \widehat{\nablal \G _l}(n) ] \right)\\
&& \hskip 2 em \approx
h_0[\a^\ep(l n_\ep)] + 
\sum_{l=1}^L  ( \bar\la(\ep) +\rho \G[\a^\ep(l n_\ep)] )
h_l[\a^\ep(l n_\ep)]
+  \kappa[\a^\ep(l n_\ep)  ],
\end{eqnarray*} 
which establishes \eq{eq:averaging}. 
Define now a piecewise constant function (on the $\delta_\ep$-subintervals):
\[
{\cal G}^\ep (\a^\ep(t), \bar\la(\ep)) = h_0[\a^\ep(l n_\ep)] + 
\sum_{l=1}^L  ( \bar\la_l(\ep) +\rho \G[\a^\ep(l n_\ep)] )
h_l[\a^\ep(l n_\ep)]
+  \kappa[\a^\ep(l n_\ep)  ],
\]
for $\delta_\ep \le t < (l+1) \delta_\ep$, then \eq{eq:averaging} implies
that:
%\begin{equation}
$\esp [\a^\ep(t+s) - \a^\ep(t) \given \Field^\ep(t)] \approx 
\int _t ^{t+s} {\cal G}^\ep[\a^\ep(s),\bar\la(\ep)] \, ds 
%\label{martingale}
%\end{equation}
$
which implies that the limit process is a martingale with zero quadratic variation. For a
detailed presentation of this methodology the reader is referred to \cite{KY03}. Taking
now the limit along the weakly convergent subsequence,
$\a^\ep(t)
\to\a(t),
\bar\la(\ep)\to\bar \la$ establishes the limiting ODE for this 
subsequence.  

\end{proof}

\end{document}